\documentclass[reqno,11pt]{preprint}
\usepackage[utf8]{inputenc}
\usepackage[english]{babel}
\usepackage[toc,page]{appendix}

\RequirePackage[backend=bibtex,maxnames=10]{biblatex}

\usepackage[marginparwidth=1in]{geometry}
\usepackage[full]{textcomp}
\usepackage[osf]{newtxtext}
\usepackage{cabin}
\usepackage{enumerate}
\usepackage{enumitem}
\usepackage{bbm}
\usepackage{dsfont}
\usepackage[varqu,varl]{inconsolata}
\usepackage[cal=boondoxo]{mathalfa}

\usepackage{hyperref}
\usepackage{breakurl}
\usepackage{mathrsfs}
\usepackage{booktabs}
\usepackage{caption}

\usepackage{tikz}
\usetikzlibrary{decorations.pathreplacing,decorations.markings}

\usepackage{amsmath}
\usepackage{amssymb}
\usepackage{mhequ}
\usepackage{mhsymb}
\usepackage{mhenvs}
\usepackage{microtype}
\usepackage{upgreek}
\usepackage{wasysym}
\usepackage{centernot}

\usepackage{oldgerm}

\newtheorem{Assumption}{Assumption}

\renewcommand{\R}{\mathbb{R}}
\renewcommand{\H}{\mathbb{H}}
\newcommand{\Di}{\mathbb{D}}

\renewcommand{\C}{\mathbb{C}}
\renewcommand{\N}{\mathbb{N}}
\renewcommand{\Q}{\mathbb{Q}}
\newcommand{\M}{\mathbb{M}}

\renewcommand{\Z}{\mathbb{Z}}
\renewcommand{\E}{\mathbb{E}}
\renewcommand{\P}{\mathbb{P}}

\newcommand{\cF}{\mathcal{F}}

\newcommand{\cH}{\mathcal{H}}

\newcommand{\cL}{\mathcal{L}}
\newcommand{\cM}{\mathcal{M}}

\def\rw{\mathrm{w}}
\def\rb{\mathrm{b}}

\colorlet{darkblue}{blue!90!black}
\colorlet{darkred}{red!90!black}
\colorlet{darkgreen}{green!50!black}
\colorlet{darkyellow}{yellow!90!black}

\def\Rangle{\rangle\!\!\!\rangle}
\def\Langle{\langle\!\!\!\langle}
\def\RRangle{\Big\rangle\!\!\!\Big\rangle}
\def\LLangle{\Big\langle\!\!\!\Big\langle}

\definecolor{Red}{rgb}{1,0,0}
\definecolor{Blue}{rgb}{0,0,1}
\definecolor{Olive}{rgb}{0.41,0.55,0.13}
\definecolor{Yarok}{rgb}{0,0.5,0}
\definecolor{Green}{rgb}{0,1,0}
\definecolor{MGreen}{rgb}{0,0.8,0}
\definecolor{DGreen}{rgb}{0,0.65,0}
\definecolor{Yellow}{rgb}{1,1,0}
\definecolor{Cyan}{rgb}{0,1,1}
\definecolor{Magenta}{rgb}{1,0,1}
\definecolor{Orange}{rgb}{1,.5,0}
\definecolor{Violet}{rgb}{.5,0,.5}
\definecolor{Purple}{rgb}{.75,0,.25}
\definecolor{Brown}{rgb}{.75,.5,.25}
\definecolor{Grey}{rgb}{.7,.7,.7}
\definecolor{Black}{rgb}{0,0,0}
\definecolor{dr}{rgb}{0.8,0,0}
\definecolor{db}{rgb}{0,0,0.8}

\def\restriction#1#2{\mathchoice
              {\setbox1\hbox{${\displaystyle #1}_{\scriptstyle #2}$}
              \restrictionaux{#1}{#2}}
              {\setbox1\hbox{${\textstyle #1}_{\scriptstyle #2}$}
              \restrictionaux{#1}{#2}}
              {\setbox1\hbox{${\scriptstyle #1}_{\scriptscriptstyle #2}$}
              \restrictionaux{#1}{#2}}
              {\setbox1\hbox{${\scriptscriptstyle #1}_{\scriptscriptstyle #2}$}
              \restrictionaux{#1}{#2}}}

\def\restrictionaux#1#2{{#1\,\smash{\vrule height .8\ht1 depth .85\dp1}}_{\,#2}}

\newcommand{\1}{\mathds{1}}

\newcommand{\dd}{\mathrm{d}}

\let\epsilon=\varepsilon

\let\Phi=\phi
\let\phi=\varphi

\let\ss=\smallskip
\let\ms=\medskip

\title{Dimers with layered disorder}

\begin{document}

\maketitle

\vspace{-2cm}

\noindent{\large \bf   Quentin Moulard, Fabio Toninelli}

\noindent{\small 
   TU Wien, Austria\\}
\noindent\email{
quentin.moulard@tuwien.ac.at,
fabio.toninelli@tuwien.ac.at}
\newline

\bigskip\noindent
\begin{abstract}
We study the dimer model on the square grid, with quenched random edge weights. Randomness is chosen to have a layered structure, similar to that of the celebrated McCoy-Wu disordered Ising model.
Disorder has a highly non-trivial effect and  it produces an essential singularity of the free energy, with $e^{-\sqrt{{\rm distance}}}$ decay of dimer-dimer correlations, at a point of the ``liquid'' (or ``massless'') phase where the homogeneous dimer model has instead a real analytic free energy and  correlations decaying like $1/({\rm distance})^2$. Moreover,
at a point where the homogeneous model has a transition between a massive (gaseous) and massless (liquid) phase, 
the critical exponent $3/2$ (Pokrovsky-Talapov law), characteristic of the transition between the two regimes, is modified by disorder into an exponent  that ranges continuously between $3/2$ and infinity.

\end{abstract}

\bigskip

\noindent{\it Key words and phrases.}
Dimer model; quenched disorder; disorder relevance; McCoy-Wu model; Lyapunov exponent; random matrix product.


\section{Introduction}
 Disorder can have a drastic,
non-perturbative effect on phase transitions: even an arbitrarily
small amount of quenched, random impurities can discontinuously change
critical exponents or remove phase transitions altogether, especially
in low-dimensional systems. Famous examples include: the 2D Random
Field Ising Model, where a random magnetic field with arbitrarily
small variance eliminates the ferromagnetic phase transition
\cite{ImryMa,AW}, and induces exponential decay of correlations
 \cite{Ding,Aizenman};  the 2D Ising model again, where random ferromagnetic
couplings conjecturally make the ferromagnetic transition 
smoother than in the homogeneous case \cite{DD,DD1}, turning the
$\log[1/(\beta-\beta_c)]$ divergence of the specific heat into a
$\log \log[1/(\beta-\beta_c)]$ one, and the spin-spin critical
exponent from $1/4$ to $0$; the pinning/wetting model, where disorder has a
smoothing effect and makes the localization transition always at least
second order \cite{giacomin2006smoothing}, in contrast with the
homogeneous case; the lattice Gaussian Free Field, where interaction with a quenched disordered substrate turns the wetting (localization) transition from first to second order in dimension $d\ge3$ \cite{HubertGB-GFF} and from (almost) first to infinite order in dimension $d=2$ \cite{HubertGFF};
the directed polymer in random
environment, where in low dimension ($d=1,2$) the polymer is
``localized'' for any arbitrarily small disorder strength
\cite{Comets} while in dimension $d\ge 3$ an infinite-order phase
transition between weak and strong disorder phases appears
\cite{lacoin2025localization}.  Despite these examples and certain
general guiding heuristic principles such as the ``Harris criterion''
\cite{harris}, the question of \emph{relevance} (in the
Renormalization Group sense) of quenched randomness on phase
transitions is still far from understood in general.

In this work, we address this question in the framework of the dimer
model \cite{Kasteleyn,mccoy1973two,Gorin} on the (periodized) square grid $\Z^2$. In contrast with the
Ising model, that in its homogeneous version has a single critical
point separating the uniqueness and phase transition region, the phase
diagram of the dimer model  exhibits a whole critical
region where correlations decay polynomially (like the inverse of the distance squared) and the height function
scales to the Gaussian Free Field (GFF) \cite{kenyon2001dominos}. In the dimer model
literature \cite{KOS}, this is called ``liquid'' or ``massless''
phase. The model exhibits also massive (or ``gaseous'') phases where
correlation decay is exponential, as well as “frozen” regions where
correlations are zero. The transition between liquid and either
gaseous or frozen regions is characterized by a critical exponent
$3/2$ (Pokrovsky-Talapov law \cite{PT,KO,Duse}). Our goal is to
understand how randomness affects this picture.

The type of
randomness we introduce is inspired by the McCoy-Wu 2D random Ising
model \cite{McCoyWu196,mccoy1973two}, where ferromagnetic spin-spin couplings are
random but constant along rows of $\Z^2$. Similarly, we randomize the
edge weights of the dimer model, in a layered way. For the McCoy-Wu
model it is expected that disorder turns the second-order phase
transition at $\beta_c$ into an infinite-order one (for recent
progress, supporting but not fully proving this conjecture, see
\cite{comets2019continuum,Quentin}. Some of the ideas we use in this work were
developed  by one of us in the latter reference).

For the dimer model, the effect of disorder is in a
sense even more surprising. Our first result (Theorem \ref{th:F} ) shows that an infinite-order phase transition appears at a point of the phase
diagram where the free energy of the homogeneous model is real
analytic. At this particular point, dimer-dimer correlations of the disordered model decay
stretched-exponentially fast (approximately like the exponential of minus the \emph{square root of the distance}) in the direction perpendicular to the
disorder layering (Theorem \ref{th:corr}). We emphasize that this
effect is present even for an arbitrarily small disorder strength, and
that it has no analog for the homogeneous model, where correlations
decay either polynomially or exponentially anywhere in the phase
diagram. It is interesting to mention that an analogous $\exp(-\sqrt{{\rm distance}})$ decay has been predicted in the physics literature \cite{Shankar} 
for energy-energy correlations of the McCoy-Wu model.
Secondly, the phase diagram of the disordered dimer model (that is parametrized by two ``magnetic fields'' $H_1,H_2$ that suitably fix the edge occupation densities) exhibits regions, that have no analog for the homogeneous model, where the free energy is affine with respect to just one of the two magnetic fields (Theorem \ref{th:Fbetane0}  and Remark \ref{rem:notaffine}).
Yet another striking finding is that, at a specific point of the phase diagram where the homogeneous model is at the 
boundary between liquid and gaseous phases, the critical exponent $3/2$
is modified into a continuously varying exponent, that ranges between
$3/2$ and $+\infty$ (Theorem \ref{th:congas}). On the other hand,
other features of the homogeneous dimer model, such as the existence of a
massless phase with correlations decaying like the inverse squared distance, and the $3/2$
exponent at the liquid/frozen boundary, turn out to be robust to the
presence of disorder (Theorems \ref{th:32} and \ref{th:corr}).

Technically, a first step of our analysis is to rewrite, via
Kasteleyn's theory, both free energy and correlations of the
disordered dimer model in terms  of a product of an i.i.d. sequence of $2\times 2$ complex random matrices
$\{M^\theta_y\}_{y\in\N}$, with $\theta\in[0,\pi/2]$ a Fourier (momentum) variable. In particular, the free energy can be expressed through an integral over $\theta$ involving  the top Lyapunov exponent
${\rm Lyap}(M^\theta)$. As realized earlier on by Derrida-Hilhorst \cite{DH} and Nieuwenhuizen-Luck \cite{Luck},
the top Lyapunov exponent has a very singular behavior in situations
where the matrices $M_y$ ``almost commute'', which is especially
relevant at the point of essential singularity and at the gas-liquid
transition point. Obtaining sharp bounds on ${\rm Lyap}(M)$
in such situations is a notoriously subtle issue, and we rely on
recent beautiful developments in this area, in particular
\cite{collin2024large,de2024scaling,collin2025lyapunov} and \cite{GGG,Havret}, the latter
making the Derrida-Hilhorst predictions rigorous. In addition, we prove in Section \ref{sec:Lyapunov}  several crucial monotonicity properties of the Lyapunov exponent, seen as a function defined on the complex half-plane. While results on the
free energy's regularity are based on fine analytic properties of
${\rm Lyap}(M)$, for the study of correlations we need arguments of different,
more probabilistic nature. 

While our analysis focuses on the model defined on the torus, we
expect the free energy singularities described by Theorems \ref{th:F}
and \ref{th:congas} to have a direct consequence on limit shapes in
domains with boundaries. In fact, it is known \cite{KO} that the free
energy as a function of the magnetic fields $(H_1,H_2)$ can be seen as
a limit shape for a volume-tilted version of the dimer model
\cite{Gorin}. For instance, we expect limit shapes to be
non-analytic around points where the slope is zero (which
corresponds to the point $(H_1,H_2)=(0,0)$ in the phase diagram, where
the essential singularity occurs). It would be very interesting to
verify this numerically.  Finally, our work raises a set of very interesting
problems, among which: the behavior of the model in the regions
$C_\pm,G_\gamma$ (see Fig. \ref{fig:phasediag},
\ref{fig:phasediag-gas}), that seems to be different from that of the
homogeneous system in any of its three phases; the scaling
limit of the height function of the disordered dimer model in the
liquid region; last but not least, the question whether the thermodynamics of the dimer model with bulk disorder (i.e., i.i.d., non-layered, random edge weigths) shows the same phenomenology as the model we study here.

{\bf Note} While this work was being completed, the work
\cite{dominosrandom} was posted on ArXiv. This also studies a
two-dimensional dimer model on the Aztec diamond with random
edge weights having a layered structure, though different from the one described in next section. In terms of both results
and techniques, the two works are quite disjoint.

\subsection{The model and the free energy}

Let $T_{L, N} \eqdef (\Z/2 L \Z) \times (\Z/2 N \Z)$ be the discrete
two-dimensional toroidal graph, for integers $L, N \geq 1$, equipped
with its natural nearest-neighbour graph structure. For convenience,
we assume that $L$ is even and $N$ is odd.  Vertices are assigned
coordinates $(x,y)$ with $1\le x\le 2L$ and $-N<y\le N$.  We color the
vertex $(x,y)$ of $T_{L, N}$ black if $x+y$ is odd and white else. We
let $V_B$ (resp. $V_W$) denote the set of black (resp. white)
vertices.  We weight the (unoriented) edges of $T_{L,N}$ in a layered
way, as follows.  Let $ {w}= (w_1,w_2)$ be a pair of positive random
variables and
${\underline{w}}=((w_1(y),w_2(y))_{y\in\mathbb Z}$ be an i.i.d. sample from ${w}$
(defined on a probability space $(\Omega, \cF, \mathbb P)$).

\begin{definition}
\label{def:w}
Let $H_1,H_2\in \mathbb R$. For an horizontal edge $e = \{(x,y),(x+1,y)\}$ let $w(e)=w_1(y)e^{-H_2}$ (resp. $w_1(y)e^{H_2}$) if $(x,y)$ is black (resp. white); for a vertical edge $e = \{(x,y),(x,y+1)\}$ let $w(e)=w_2(y)e^{H_1}$
(resp. $w_2(y)e^{-H_1}$) if $(x,y)$ is black (resp. white).

\end{definition}
The parameter $H_1$ plays the role of a ``magnetic field'' that fixes the average difference between the number of vertical edges with black or white bottom vertex or, equivalently, the average height change in the horizontal direction \cite{KOS}. Similarly,
  $H_2$  fixes the average difference between the number of horizontal edges with white or black left vertex, or equivalently the average height change in the vertical direction. We do not formally define the dimer model's height function \cite{Gorin}, since we will not use it in this work.

\begin{Assumption}
  \label{ass:disordine}
The distributions of $w_1,w_2$ are compactly supported on $(0,+\infty)$, $w_2$ is not deterministic and $w_1,w_2$ are independent. Up to multiplying all weights by a global positive constant, which does not change the measure and changes the free energy by an additive constant, we will assume that $\mathbb E(\log w_2)=0$.
\end{Assumption}

The compact support assumption   avoids  technical complications and could be relaxed to some extent, provided that the the logarithm of the weights has  finite moments of sufficiently high order.

The distribution  $P_{L,N}:=P_{L, N, \underline{w},H_1,H_2}$ of the disordered dimer model is the probability measure whose density with respect to the uniform measure on the set of $\Omega_{L,N}$ of perfect matchings $D$ of $T_{L,N}$ is proportional to $\prod_{e \in D} w(e)$. Let $Z_{L, N}$ be the  partition function
\begin{equ} \label{e:PartitionDef}
Z_{L, N}:=Z_{L, N, \underline{w},H_1,H_2} \eqdef \sum_{D\in\Omega_{L,N}} \prod_{e \in D} w(e) \, .
\end{equ}

Let us define the free energy per unit volume in the thermodynamic limit as
\begin{equ} \label{e:FreeEnergyLimitDef}
F\eqdef F(H_1,H_2)\eqdef \lim_{L \to \infty} \lim_{N\to\infty}\frac{1}{4 L N } \log Z_{L,N} \, .
\end{equ}
$F$ can be expressed via the top Lyapunov exponent of a product of i.i.d. $2\times 2$ random matrices:
\begin{proposition} \label{p:FreeEnergyFormula}
The limit \eqref{e:FreeEnergyLimitDef} exists $\mathbb P$-almost surely and in $\mathbb L^1$, and it is deterministic:
\begin{equ} \label{e:FreeEnergyFormula}
F = \frac{1}{\pi} \int_{0}^{\pi/2} \max\Big[\cL(\theta,H_2 ),|H_1|\Big] \, \dd \theta \, ,
\end{equ}
where,
for $\theta,H_2 \in \mathbb R$ and $w_1,w_2 > 0$, we define the $2 \times 2$ complex matrix
\begin{equ}
M^\theta_{H_2}\eqdef M_{H_2}^\theta(w_1,w_2) \eqdef
\begin{pmatrix} \label{e:Matrix}
2 w_1 \sin(\theta + \iota H_2) & w_2^2 \\
1 & 0
\end{pmatrix} \, ,
\end{equ}
and
$\cL(\theta,H_2)\ge0$ stands for the top Lyapunov exponent of the product of the i.i.d. matrices $\{M^\theta_{y,H_2}\}_{y\in \N} \eqdef \{M^\theta_{H_2}(w_1(y),w_2(y))\}_{y\in \N}$.
\end{proposition}
The definition of top  Lyapunov exponent of a product of i.i.d. random matrices is recalled in Section \ref{sec:Lyapunov} below. We refer also to the monographs \cite{Bougerol,Viana}.
\ss

For later convenience, we define also, given $H_2\in\R$,
\begin{eqnarray}
  \label{eq:bstar}
  \Delta (H_2) := \cL(0,H_2) \, , \quad H_c(H_2) :=\cL\Big(\frac\pi2,H_2\Big) \, ,
\end{eqnarray}
and whenever $H_2$ is omitted in the notations $F(H_1, H_2)$,
$\cL(\theta, H_2)$, $M_{H_2}^\theta$ or $H_c(H_2)$, it means that
$H_2$ is chosen to be zero. Note that for $H_2 = 0$, $\Delta(0) = 0$
and the entries of the matrix $M^\theta$ are real and non-negative:
\begin{equ}
  \label{e:Matrixnew}
M^\theta = \begin{pmatrix} 
2 w_1 \sin(\theta) & w_2^2 \\
1 & 0
\end{pmatrix} \, .
\end{equ}

\begin{remark}
  In \eqref{e:FreeEnergyLimitDef} we take the limit $N\to\infty$
  before $L\to\infty$ for technical simplicity, but with some extra
  work one could take the limit in any order and the result would be the
  same.
  On the other hand, it is crucial that the model is defined on
  the torus. In fact, if $T_{L,N}$ were replaced by the cylinder
  $(\mathbb Z/2L\mathbb Z)\times \{-N+1,\dots,N\}$, the result would be
  very different and less interesting: in this case, the infinite-volume free energy is
\begin{equ} 
F = \frac{1}{\pi} \int_{0}^{\pi/2} \cL(\theta,H_2) \, \dd \theta \, ,
\end{equ}
  and in particular it is independent of $H_1$. The reason is that, for every perfect matching of the cylinder, each row contains exactly as many vertical edges with black vertex on top as vertical edges with white vertex on top, so that $H_1$ cancels out in the weight of configurations.
\end{remark}

\begin{remark}
  \label{rem:restrict}
 $F$ is a function of $H_1,H_2$ and of the law of $w_1,w_2$, but we think of  the latter as fixed and we write $F(H_1,H_2)$.
 $F$ is even-symmetric in $H_1$ and in $H_2$ (which can  be seen directly by translating the model by $(1,0)$ or $(0,1)$). Therefore, we will restrict our attention to $H_1,H_2\ge0$.
\end{remark}

   In absence of disorder (that is $w_1>0$ deterministic and $w_2=1$) the Lyapunov exponent  and the free energy (that we denote $\mathbf F$ to distinguish it from $F$) can be computed easily. In fact,  one can more simply write $\mathbf F$ as a double integral on the torus $\{z,w\in \C:|z|=|w|=1\}$ \cite[Th. 3.5]{KOS}. The qualitative outcome of the  computation is summarized in Figure \ref{fig:phasediag} (left) and its caption.
 In particular, for $H_2=0$, an explicit  computation (Appendix \ref{app:pure}) gives
   \begin{equ}
     \label{32pure}
{     \mathbf F}(H_1)-\frac{H_1   }2
=     \begin{cases}
         r(w_1)(\mathbf H_{c}-H_1)^{3/2}(1+o(1)) & \mathrm{if }\;       H_1\uparrow \mathbf H_{c}\\
       0& \mathrm{if }\;       H_1\ge \mathbf H_{c}
     \end{cases}
   \end{equ}
  where $r(w_1),\mathbf H_{c}>0$ are  explicit functions of $w_1$, while for $H_1\in(-\mathbf H_{c},\mathbf H_{c})$ the function ${     \mathbf F}(H_1)$ is analytic and is quadratic for $H_1\sim0$. The regions $(-\mathbf H_{c},\mathbf H_{c})$ and $\{H_1:|H_1|\ge \mathbf H_{c}\}$ are the intersections of the ``liquid'' and ``frozen'' regions (see Fig. \ref{fig:phasediag} (left)) with the axis $H_2=0$.
The power-law singularity with exponent $3/2$  takes the name of Pokrovsky-Talapov law \cite{PT}, and the exponent is the same at any point of the liquid/frozen boundary.

\subsection{Infinite-order phase transition induced by disorder}

As we shall see later (beginning of Section \ref{s:proofs}), the function $\theta \in [0,\pi/2] \to \cL(\theta, H_2)$ is non-negative, strictly increasing, continuous, and it is real analytic on $(0,\pi/2]$. In particular, the inverse $\cL^{-1}(\cdot, H_2)$ is well-defined and we have
\begin{equ}
  \label{formulaF0}
     F(H_1, H_2) =
     \begin{cases}
       F(0, H_2) & \text{ if } 0 \leq H_1 \leq \Delta(H_2) \\
       F(0, H_2)+\frac1\pi\int_{\Delta(H_2)}^{H_1}
       \cL^{-1}(y,H_2) \, dy & \text{ if } \Delta(H_2) \leq H_1 \leq H_c(H_2) \\
       \frac{H_1}{2} & \text{ if } H_1 \ge H_c(H_2)
     \end{cases} \, .
   \end{equ}
Note that $F(\cdot, H_2)$ is strictly convex and real analytic on $(\Delta(H_2), H_c(H_2))$.

\begin{figure}[h]
    \centering
    \begin{minipage}{0.45\textwidth}
        \centering
        \includegraphics[width=0.9\textwidth]{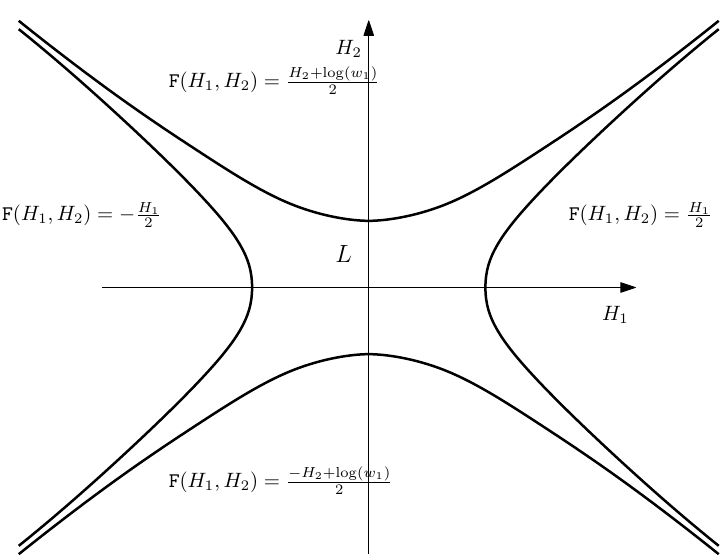} 
    \end{minipage}\hfill
    \begin{minipage}{0.45\textwidth}
        \centering        \includegraphics[scale=0.66]{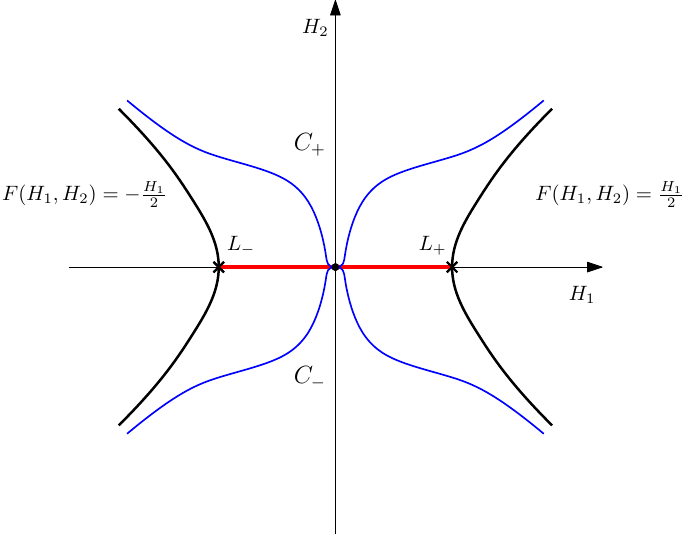} 
    \end{minipage}
    \caption{Left: The phase diagram of the non-disordered dimer
      model. The free energy ${\mathbf F}$ is a strictly convex and
      real analytic function of $(H_1,H_2)$ in the ``liquid'' (or
      ``massless'') region $L$, and it is affine in the four
      ``frozen'' (or ``solid'') unbounded convex regions. In each
      solid region, in the thermodynamic limit, the system is frozen
      in a single periodic configuration where only one of the four
      types of edges is occupied by a dimer.  Right: A sketchy picture
      of the phase diagram of the disordered dimer model. At
      $(H_1,H_2)=(0,0)$ (thick dot) the free energy $F$ has an
      infinite-order phase transition (Theorem \ref{th:F}) and
      correlations decay as a stretched exponential (Theorem
      \ref{th:corr}).
On the two open red segments
      (Theorem \ref{th:corr}) and, conjecturally, in the whole liquid
      regions $L_\pm$, correlations decay
      like the inverse distance squared.
In the two unbounded regions $C_\pm$
      above/below the blue curves, $F$ is constant in the horizontal
      direction (Theorem \ref{th:Fbetane0}). The regions
      $C_\pm$ are in general \emph{not convex} (Remark \ref{rem:notaffine}). 
      For $|H_1|$ large enough  the model is in a solid
      phase where $F(H_1,H_2)=|H_1|/2$     (formula \eqref{e:FreeEnergyFormula}) and correlations are exactly zero (Remark \ref{rem:frozen}).
      The crosses denote the
      points $(\pm H_c,0)$ on the liquid/frozen transition lines. The free energy critical exponent is $3/2$ on the whole liquid/frozen transition lines (thick black lines), see Theorem \ref{th:32}.
      Possibly, the model is
      frozen also deep inside the regions $C_\pm$.  }
    \label{fig:phasediag}
\end{figure}

Our first main results (Theorems \ref{th:F} and \ref{th:Fbetane0}) show that the free energy of the disordered dimer model has an essential singularity at $H_1=H_2=0$, which has no analog for the pure dimer model. On the other hand, the critical exponent of the liquid/solid transition  is unchanged by the presence of randomness (Theorem \ref{th:32}).

\begin{theorem} \label{th:F} Under Assumption \ref{ass:disordine} and for $H_2 = 0$, it holds as $H_1 \searrow 0$,
\begin{eqnarray}
    \label{eq:Fessential}
    F(H_1)-F(0)=\exp\left(-(1+o(1)) \, \frac{{\rm Var}(\log w_2)}{H_1}\right).
\end{eqnarray}
\end{theorem}

\begin{remark}
Under the additional assumption that the law of $\log w_2$ admits a density with respect to Lebesgue's measure, \eqref{eq:Fessential} can be strengthened to
   \begin{eqnarray}
     \label{eq:Fessential2}
     F(H_1)-F(0)\stackrel{H_1\searrow 0}\asymp H_1 \exp\left(-\frac{{\rm Var}(\log w_2)}{H_1}\right),
   \end{eqnarray}
i.e., that the ratio of right- and left-hand sides is bounded away from zero and infinity as $H_1\searrow 0$.
\end{remark}

It may look surprising that disorder produces a singularity at a point where the free energy of the homogeneous model is analytic. However, see Remark \ref{rem:incontext} below.

In order to better understand  the essential singularity at $H_1=0$, it is useful to view $F$ as a function of both magnetic fields, hence to consider the case $H_2 \ne 0$.  Disorder produces another phenomenon that has no analog for the pure dimer model: in the $(H_1,H_2)$ plane there are two regions, that are tangent exactly at the essential singularity point, where the free energy is constant in $H_1$. Indeed, contrary to the pure dimer model, here $\Delta(H_2) > 0$ as soon as $H_2 \ne 0$.

\begin{theorem} \label{th:Fbetane0}
$\Delta(H_2)$ is the top Lyapunov exponent of the product of i.i.d. random matrices $V_y^{2\sinh(H_2)} = V^{2\sinh(H_2)}(w_1(y), w_2(y))$ where for $x \in \R$ we denote
\begin{eqnarray}
\label{eq:D2V}
V^x =V^x(w_1,w_2)= \begin{pmatrix}
x \,{w_1/w_2} & -w_2\\
{1/w_2} & 0
\end{pmatrix}.
\end{eqnarray}
Under Assumption \ref{ass:disordine}, $\Delta(H_2) > 0$ as soon as $H_2\ne0$ and
\begin{equation} \label{eq:betacgamma}
\Delta(H_2)\stackrel {H_2\searrow 0}\sim \frac{{\rm Var}(\log w_2)}{|\log H_2|} \, .
\end{equation}

\end{theorem}

\begin{remark}
  \label{rem:notaffine}
In contrast with frozen/gaseous regions of the homogeneous model, that are always convex, the region $C_+\eqdef \{(H_1,H_2):H_1\in(-\Delta(H_2),\Delta(H_2)),H_2>0\}$ is not convex, at least if the variance of the disorder is small enough. Indeed, since the free energy $F$ tends to $\bf F$ as  ${\rm Var}(\log w_2)$ tends to zero, the  boundary of $C_+$ tends to the upper liquid/frozen boundary in
      Fig. \ref{fig:phasediag} (left), plus a vertical segment connecting it to
      $(0,0)$. As a consequence, at least for ${\rm Var}(\log w_2)$ small enough,  the convex function $F$ cannot be affine in both variables in the whole region $C_+$. In this sense, the nature of this region is qualitatively different from that of   frozen or gaseous regions of the homogeneous model.
\end{remark}

The last result in this section addresses the free energy's behavior at the liquid/frozen boundary, i.e. for $H_1\to H_c(H_2)$, see definition \eqref{eq:bstar}. The result is that disorder does not modify the Pokrovsky-Talapov exponent $3/2$ of the liquid/frozen transition:
\begin{theorem}
  \label{th:32}
  For any $H_2\in\R $ there exists    $r=r(H_2)\in(0,\infty)$ such that
  \begin{eqnarray}
    \label{32dis}
  F(H_1,H_2)-\frac{H_1}2\stackrel{H_1 \nearrow H_c(H_2)} =
       r\,(H_c(H_2)-H_1)^{3/2} \, (1+o(1)) \, .
  \end{eqnarray}
\end{theorem}

See Fig. \ref{fig:phasediag} (right)
for a pictorial illustration of the contents of Theorems \ref{th:F}, \ref{th:Fbetane0} and \ref{th:32}.

\subsection{The liquid-gas transition}
 \label{sec:congas}
In addition to frozen and liquid phases, the (non-disordered) dimer model on  bipartite periodic graphs can  exhibit so-called ``gaseous'' (or ``massive'' phases) \cite{KOS,Duse}. These are bounded, convex regions of the $(H_1,H_2)$ plane where the free energy is affine and dimer-dimer correlations decay exponentially fast. In this section, we investigate the effect of disorder on the liquid/gas transition.

The model discussed so far, in the non-disordered version, does not exhibit a gaseous phase. However, this can be easily fixed by defining the weights  $w$ as in Definition \ref{def:w}, except that $w_2(y)$ is replaced by
$w_2(y) \, \gamma^{2(y\!\!\mod 2)-1}$ with $\gamma > 1$. That is, vertical weights in rows of even (resp. odd) index are multiplied by $\gamma$ (resp. $1/\gamma$). The choice $\gamma>1$ is just a convention, since a translation of the model by $(0,1)$ turns $\gamma$ into $1/\gamma$.
In this case the pure model ($w_1$ constant, $w_2=1$) has indeed a gaseous phase, that shrinks to zero if $\gamma\to 1$  (see Appendix \ref{app:pure} and Fig. \ref{fig:phasediag-gas} (left)).
\begin{figure}[h]
    \centering
     \begin{minipage}{0.45\textwidth}
         \centering
        \includegraphics[width=0.85\textwidth]{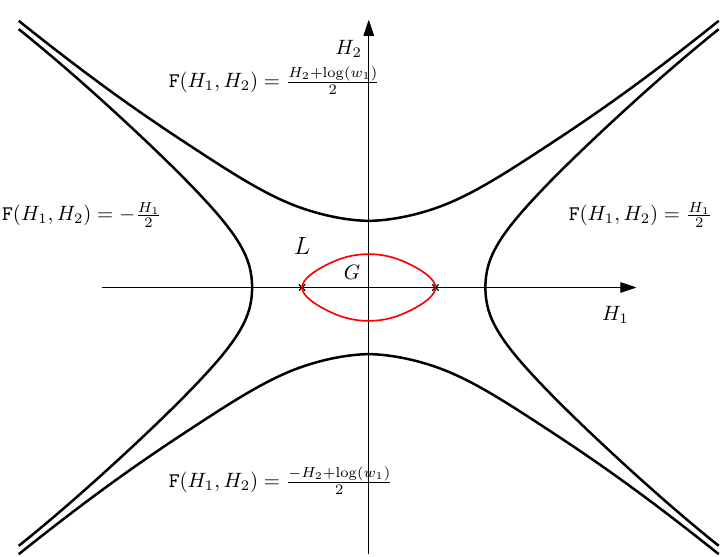} 
     \end{minipage}\hfill
     \begin{minipage}{0.45\textwidth}
         \centering        \includegraphics[scale=0.95]{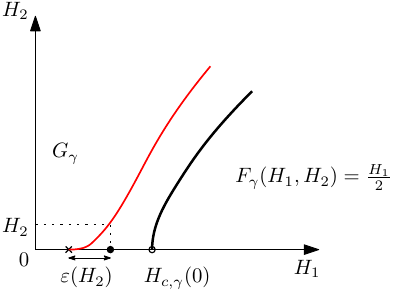} 
     \end{minipage}
        \caption{Left: The phase diagram of the non-disordered dimer
          model with $\gamma>1$. In the gaseous phase $G$, the free
          energy ${ \mathbf F}_\gamma(H_1,H_2)$ is affine and
          correlations decay exponentially fast. The two crosses mark
          the points $(\pm\log \gamma,0)$. At the liquid/gas and
          liquid/frozen boundaries, ${ \mathbf F}_\gamma$ has a $x^{3/2}$
          singularity \cite{KO,Duse}.   Right: The phase diagram of the non-disordered dimer
          model with $\gamma>1$ (we restrict by symmetry to $H_1,H_2\ge0$). To the right of the black curve, that starts at $(H_{c,\gamma}(0),0)$, the system if frozen. In the region $G_\gamma$ the free energy $F_\gamma$ is constant with respect to $H_1$. The boundary of $G_\gamma$ intersects the horizontal axis at the point  $(\log \gamma,0)$, marked with a cross.  For $\eps$ small, $F_\gamma(\log \gamma+\eps,0)\asymp \eps^{\max(1+\frac1{2\alpha(\gamma)},\frac32)}$ (Theorem \ref{th:congas}; see Eq. \eqref{eq:agamma} for $\alpha(\gamma)$). We conjecture that the boundary of $G_\gamma$ is singular for $H_2\to0$, more precisely that the relation between $H_2$ and $\eps $ is $\eps(H_2)=\Delta_\gamma(H_2)-\log \gamma\approx c(\alpha(\gamma))(H_2)^{2\alpha(\gamma)\wedge 1}$ with $c(\alpha(\gamma))$ positive if $\alpha(\gamma)<1/2$ and negative if $\alpha(\gamma)>1/2$ (Remark \ref{rem:conj} and Section \ref{sec:conj}).}
    \label{fig:phasediag-gas}
\end{figure}
Whenever $(H_1,H_2)$ approaches (from the liquid region) a generic point  on the boundary of either the  gaseous or of a  frozen region, where ${     \mathbf F}_\gamma$ is an affine function $A(H_1,H_2)$,  the free energy ${     \mathbf F}_\gamma$ has a $x^{3/2}$ singularity \cite[Th. 1.5]{Duse}, in the sense that $ {     \mathbf F}_\gamma(H_1,H_2)- A(H_1,H_2)\approx D^{3/2}$, with $D$ the distance between  $(H_1,H_2)$  and the liquid region's boundary. This behavior takes the name of Pokrovsky-Talapov law \cite{PT} in the physics literature. In this sense, in view of \eqref{32pure}, in absence of disorder, frozen/liquid and gas/liquid boundaries show the same qualitative local properties.

\begin{remark}
  \label{rem:incontext} While for $\gamma=1$ the free energy $(H_1,H_2)\mapsto \bf F$ of the homogeneous model is analytic around $(H_1,H_2)=0$, when considered as a function of the \emph{three} parameters $H_1,H_2\in\R,\gamma\ge 1$,  $\bf F_\gamma$ is \emph{not} analytic at $(H_1,H_2)=0,\gamma=1$, because the gasous phase that shrinks to a point for  $\gamma\to1^+$. This may help explain why for $\gamma=1$ disorder produces an essential singularity at $(H_1,H_2)=0$.
\end{remark}

As we will see, \emph{the situation is very different in presence of disorder.}
First of all, in analogy with Proposition \ref{p:FreeEnergyFormula}, one can write the free energy in terms of the top Lyapunov exponent of a product of random matrices. Call $Z_{\gamma,L,N}$ the disorder-dependent partition function. Then, the limit $F_\gamma(H_1,H_2)$ defined as in \eqref{e:FreeEnergyLimitDef} with $Z_{L,N}$ replaced by $Z_{\gamma,L,N}$, exists $\P$-almost surely and in $\mathbb L^1$, and equals (see  Remark \ref{rem:Fgamma})
\begin{eqnarray}
  \label{eq:FgammaFormula}
  F_\gamma(H_1,H_2)= \frac{1}{\pi} \int_{0}^{\pi/2} \max\Big[\cL_\gamma(\theta,H_2 ),|H_1|\Big] \, \dd \theta \, ,
\end{eqnarray}
where $\cL_\gamma(\theta,H_2)\ge0$ is $1/2$ times the top Lyapunov exponent of the product of the i.i.d. copies of
\begin{eqnarray}
  \label{Mgamma}
  T^\theta_\gamma\eqdef
  \begin{pmatrix}
    2w_1 \sin(\theta + \iota H_2) & (\gamma w_2)^2 \\
1 & 0
\end{pmatrix} \begin{pmatrix}
    2w_1' \sin(\theta + \iota H_2) & (w'_2/\gamma)^2 \\
1 & 0
\end{pmatrix}
\end{eqnarray}
with $w_1',w_2'$ independent copies of $w_1,w_2$.
As is the case for $\gamma=1$, $F_\gamma$ is even-symmetric in $H_1,H_2$ and  we restrict to non-negative values of the magnetic fields. 
Like in the case $\gamma=1$, it is easy to see that $H_1\mapsto
F_\gamma(H_1,H_2)$ is constant for $H_1\in [-\Delta_\gamma(H_2),\Delta_\gamma(H_2)]$ with
\begin{eqnarray}
  \label{eq:Delta2g}
\Delta_\gamma(H_2):=\min\Big\{\cL_\gamma(\theta,H_2):\theta\in\Big[0,\frac\pi2\Big]\Big\}=\cL_\gamma(0,H_2)
\end{eqnarray}
(the function $\cL_\gamma$ is non-decreasing in $\theta$,  see also Remark \ref{rmk:Lgamma}).
The effect of disorder is particularly striking around the point $(H_1,H_2)=(\Delta_\gamma(0),0)=(\log\gamma,0)$, as the next result shows:
\begin{theorem}
  \label{th:congas}
  Let $H_2=0$ and $\gamma>1$, and  $F_\gamma(H_1)\eqdef F_\gamma(H_1,0)$. Under Assumption \ref{ass:disordine},
  define
  \begin{eqnarray}
    \label{gammac}
  \gamma_c=\Big(\mathbb E(w_2^2) \, \mathbb E (w_2^{-2})\Big)^{1/4}>1
  \end{eqnarray}
  and $\alpha(\gamma)$ as the unique strictly positive solution of
  \begin{eqnarray}
    \label{eq:agamma}
    \gamma^{4\alpha(\gamma)}=\mathbb E[(w_2)^{2\alpha(\gamma)}] \, \mathbb E[(w_2)^{-2\alpha(\gamma)}]
  \end{eqnarray}
  or $\alpha(\gamma)=+\infty$ if no positive solution exists.
The free energy satisfies
  \begin{eqnarray}
    \label{eq:Fgas}
    F_\gamma(H_1)-F_\gamma(0)=
    \begin{cases}
      \label{eq:o1}
      0 & \text{if }  H_1\le \log \gamma\\
      (H_1-\log\gamma)^{\beta(\gamma)+o(1)} & \text{if }  H_1\searrow \log \gamma
    \end{cases}
  \end{eqnarray}
  where
  \begin{eqnarray}
    \label{eq:betag}
  \beta(\gamma)=\max\left(1+\frac1{2\alpha(\gamma)},\frac32\right).
  \end{eqnarray}
The function  $\gamma \in  (1,\infty) \mapsto \beta(\gamma)\in [3/2,\infty)$ is  continuous and  non-increasing, with $\beta(\gamma)\uparrow \infty$ for $\gamma\searrow 1$ and $\beta(\gamma)=3/2$ for $\gamma\ge \gamma_c$.
\end{theorem}
In other words, according to the value of $\gamma$, the order of the  phase transition (that occurs at $H_1=\log \gamma$ along the $H_2=0$ axis) varies continuously, ranging from the Pokrovsky-Talapov law (exponent $3/2$) for $\gamma\ge \gamma_c$ to the infinite-order phase transition \eqref{eq:Fessential} for  $\gamma\to1$.

Uniqueness of a strictly positive solution to \eqref{eq:agamma} follows from the strict convexity of $\alpha \mapsto \E[Z^\alpha]$ if $Z$ is a non-deterministic positive random variable.
From \eqref{eq:agamma},   \eqref{eq:betag}  and $\E(\log w_2)=0$ one sees that
\begin{eqnarray}
  \label{eq:ag2}
  \alpha(\gamma)\stackrel{\gamma\to1^+}\sim \frac{\log \gamma}{{\rm Var}(\log w_2)}, \quad \beta(\gamma)\stackrel{\gamma\to1^+}\sim \frac{{\rm Var}(\log w_2)}{2\log \gamma}.
\end{eqnarray}

\begin{remark}
  If either $\gamma<\gamma_c$  and the law of $\log w_2$ has a continuously differentiable density, or if $\gamma>\gamma_c$ (without other assumptions), then instead of  \eqref{eq:o1} we have the sharper asymptotic
\begin{eqnarray}
  \label{eq:o2}
  F_\gamma(H_1)-F_\gamma(0)\stackrel{H_1\searrow \log \gamma}\asymp (H_1-\log \gamma)^{\beta(\gamma)}.
\end{eqnarray}
\end{remark}
See Fig.  \ref{fig:phasediag-gas} (right) for an illustration of Theorem \ref{th:congas}.

\begin{remark}\label{rem:conj}
  The singular behavior \eqref{eq:o1} of the free energy  is related to the fact that the curve $H_2\mapsto \Delta_\gamma(H_2)$ is singular around $H_2=0$, where it meets the point $(H_1,H_2)=(\log \gamma,0)$. Indeed, we conjecture that
  \begin{equation}
    \label{eq:conj}
  \Delta_\gamma(H_2)-\log\gamma
\stackrel{H_2\to0}\sim c_{\alpha(\gamma)} |H_2|^{2\min(\alpha(\gamma),1)},
  \end{equation}
with $c_\alpha>0$ if $\alpha<1/2$ and $c_\alpha>0$ if  $\alpha<1/2$.   See Section \ref{sec:conj} for a the proof of a weaker (i.e., integrated) version of this statement, see \eqref{almostlaplace}.
\end{remark}

Finally, we could also prove that the exponent $3/2$ of the liquid/frozen phase transition, that occurs at
\begin{eqnarray}
  \label{Hstargamma}
  H_{c,\gamma}(H_2)=\cL_\gamma(\pi/2,H_2),
\end{eqnarray}
is unchanged by disorder as is the case for $\gamma=1$, but we refrain from giving  a formal  proof of this, as the argument is very similar to the case $\gamma=1$.

\subsection{Dimer-dimer correlations}
Up to now, we have studied the critical behavior of the disordered dimer model only in terms of free energy. In this section, we study the counterpart of the essential singularity on correlation functions.
Therefore, we focus on $H_2=0,\gamma=1$. Let us consider two edges
$e$ and $e' $ of $T_{L, N}$. We are interested in the two-point
correlations
\begin{equ} \label{e:Cov} \mathrm{Cov}_{L, N}(e,e') := P_{L, N}(e,e' \in D) -
  P_{L, N}(e \in D) P_{L, N}(e'
  \in D) \,
\end{equ}
when the two edges are located in two faraway rows.
It turns out that, as the free energy, the correlation, as well as its thermodynamic limit
\begin{equ}
\mathrm{Cov}_{\underline{w}}(e,e'):=\lim_{L\to\infty}\lim_{N\to\infty}  \mathrm{Cov}_{L, N}(e,e')\,,
\end{equ}
can be expressed in terms of the product of the matrices $M^\theta_y$, see Proposition \ref{prop:formuloneCov}. The limit depends on the disorder realization $\underline w$, as we emphasize in the notation $\mathrm{Cov}_{\underline{w}}$.
For definiteness, we make the following assumption:
\begin{Assumption}
  \label{ass:edges}
Both edges $e,e'$ are horizontal,  $e$ has its black vertex on the left and $e'$ on the right. We write
$e = \{(x,y),(x+1,y)\}$ and $e' = \{(x',y'),(x'+1,y')\}$. Without loss of generality, by translation and reflection invariance in law of the disorder, we assume  $x'=y'=0$ and $y>0$.
\end{Assumption}

 For the pure model, ${\rm Cov}(e,e')\asymp(x^2+y^{2})^{-1}$ (times an oscillating prefactor) if $H_1\in (-\mathbf H_c,\mathbf H_c)$, see e.g. \cite[Sec. 4.4]{KOS} 
, and is zero if $|H_1|\ge \mathbf H_c$, and a similar behavior holds in the whole liquid region of the $(H_1,H_2)$ plane.
In presence of disorder, the infinite-order phase transition at $H_1=0$ translates into a stretched-exponential decay of correlations, while the inverse square distance decay persists in the liquid phase:
\begin{theorem}
  \label{th:corr} Let  the edges $e,e'$ be as in Assumption \ref{ass:edges} and let $H_2=0$.
  \begin{itemize}
  \item   
    \label{th:corrdecay}
    If $H_1=0$,
    \begin{equation}
      \label{eq:corrstretched}
\left|\mathrm{Cov}_{\underline{w}}(e,e')\right|\le       C e^{-r\sqrt y}
    \end{equation}
    where $C\in(0,\infty)$ is a constant and $r$ is a strictly positive random variable, that can depend on $y$, but is such that $\sup_y\P(r<\delta)$ tends to zero when $\delta\to0$.

    If in addition $|x|\le e^{\eps_y \sqrt y}$ and $|\cos(\frac\pi2(x \!\!\mod 4))|\ge \delta>0$ for  $\eps_y\to0$ as $y\to\infty$, then
    \begin{eqnarray}
      \label{eq:stretched4}
        (-1)^{y+1}\mathrm{Cov}_{\underline{w}}(e,e') \ge  C^{-1} \,e^{-r^{-1}\sqrt y}.
    \end{eqnarray}
    The positive constant $C$ and  random variable $r$ depend  also on $\delta$ and on the sequence $(\eps_y)_y$.
  \item
      If $H_1\in(0,H_c)$, there exists  a strictly positive and finite random variable $c_+=c_+(H_1)$ such that for all $y>0$
\begin{equation}
      \label{eq:corrstretched2}
      |\mathrm{Cov}_{\underline{w}}(e,e')|\le       \frac{c_+}{ y^2}.
    \end{equation}
    If in addition  $|x|\le \delta_1 y$ and $ |\cos(x \cL^{-1}(H_1)-\frac\pi2(x \!\!\mod 4))|>\delta_2>0$ and  sufficiently small  $\delta_1$, then
    \begin{eqnarray}
      \label{eq:stretched3}
      \mathrm{Cov}_{\underline{w}}(e,e') \ge \frac {c_-}{ y^2},
    \end{eqnarray}
where $c_-=c_-(H_1,\delta_1,\delta_2)$ is a strictly positive and finite random variable.
  \end{itemize}

\end{theorem}
We will see that the random variable $r$ can be expressed in terms of a one-dimensional Brownian motion.
Note that at the essential singularity point, correlations decay stretched-exponentially fast in the $y$ direction, but very slowly in the $x$ direction.

It is possible to derive  results analogous to those of Theorem \ref{th:corr} if one or both edges $e,e'$ are vertical, or if the colors of the vertices of one edge are swapped. 
\begin{remark}
  The restrictions in \eqref{eq:stretched4}, \eqref{eq:stretched3} on the values of $x$ are unavoidable and reflect an
  actual phenomenon. Indeed, even in the non-disordered case
  dimer-dimer correlation in the liquid region  decay like a pure power law $\asymp (|x|^2+|y|^2)^{-1}$, times an oscillating and periodically vanishing  prefactor  \cite[Sec. 4.4]{KOS}
, so
  one cannot expect a $|y|^{-2}$ behavior  uniform in $x$.
\end{remark}
\begin{remark}
  \label{rem:frozen}
For $H_1\ge H_c$, Theorem \ref{th:F} implies that $\partial_{H_1}F(H_1)=1/2$, from which one immediately deduces that, in the thermodynamic limit,   vertical edges with black vertex at the bottom are occupied with probability one, and edges of other types are occupied with density zero. That is, the system is in the frozen phase, and  dimer-dimer correlations identically vanish.
\end{remark}






\section{Action of the matrices on the Poincaré half-plane}
\subsection{Holomorphic functions on the Poincaré half-plane}

In our analysis, a special role will be played by holomorphic functions on the Poincaré right
half-plane $\mathbb H \eqdef \{z \in \C, \Re(z) > 0\}$ (we also denote
$\overline{\mathbb H} \eqdef \{z \in \C, \Re(z) \geq 0\} \sqcup
\{\infty\}$, and
$\partial \mathbb H \eqdef \overline{\mathbb H} \setminus \mathbb
H$). $\H$ admits a natural Riemannian manifold structure whose metric
is given by $\dd z \dd \overline{z}/\Re(z)^2$. Let us denote
$d_{\mathbb H}(\cdot, \cdot)$ the corresponding distance. The Schwarz-Pick theorem \cite[Th. 2.3.3]{Jost}
 states that every holomorphic function
$f \colon \mathbb H \mapsto \mathbb H$ is $1$-Lipschitz for the distance
$d_{\mathbb H}$, in the sense that for every $z_1,z_2 \in \mathbb H$
it holds $d_{\mathbb H}(f(z_1),f(z_2)) \leq d_{\mathbb
  H}(z_1,z_2)$. Thus, for $f \colon \mathbb H \mapsto \mathbb H$
holomorphic, the contraction rate
\begin{equ}
\tau_{\mathbb H}(f) \eqdef \sup_{z_1 \ne z_2 \in \mathbb H} \frac{d_{\mathbb H}(f(z_1),f(z_2))}{d_{\mathbb H}(z_1,z_2)} = \sup_{z \in \mathbb H} \frac{\Re(z)}{\Re(f(z))} |f'(z)| \, ,
\end{equ}
belongs to $[0,1]$. In fact, $\tau_{\mathbb H} (f) < 1$ as soon as the range $f(\mathbb H)$ is relatively compact in $\mathbb H$. More precisely we have the following  quantitative estimate.
\begin{lemma} \label{lem:EasyBoundContraction}
Let $f \colon \mathbb H \mapsto \mathbb H$ be a holomorphic map, and let $D \in [0, +\infty]$ be the diameter of its range $f(\mathbb H)$ for the distance $d_{\mathbb H}$. Then it holds
\begin{equ}
\tau_{\mathbb H} (f) \leq \tanh(D/2) \, .
\end{equ}
\end{lemma}

\begin{proof} Given $z\in\H$ let $\psi_{0\mapsto z}$ be a M\"obius transformation that maps the unit disk $\Di$ to $\H$ and $\psi_{0\mapsto z}(0)=z$. This is an isometry between $\H$ with the distance $d_\H$ and $\Di$ with hyperbolic distance $d_\Di$. Let us fix $z \in \mathbb H$ and prove that for every $z' \in \mathbb H$ it holds
\begin{equ}
d_\H(f(z), f(z')) \leq \tanh(D/2) \, d_\H(z,z') \, .
\end{equ}
To do so, let us define $g:\Di\mapsto \Di$ as $g=\psi^{-1}_{0\mapsto f(z)}\circ f\circ \psi_{0\mapsto z}$, so that $g(0)=0$ and
  \begin{eqnarray}
    \label{eq:HD}
    d_\H(f(z),f(\psi_{0\mapsto z}(y)))=d_\Di(0,g(y)),\quad d_\H(z,\psi_{0\mapsto z}(y))=d_\Di(0,y).
  \end{eqnarray}
Since the range of $\psi_{0\mapsto z}$ is $\H$, it is enough to prove that for every $y \in \Di$
\begin{equ}
d_{\Di}(0,g(y)) \le \tanh(D/2) \, d_{\Di}(0,y) \, ,
\end{equ}
or equivalently, since $d_\Di(0,\cdot)=2 \artanh(|\cdot|)$,
  \begin{eqnarray}
    \label{eq:HD2}
 \artanh(|g(y)|) \le (\tanh(D/2) \artanh(|y|) \, .
  \end{eqnarray}
  Note that the range $g(\Di)$ has diameter $D$ for $d_\Di$, so that by assumption for every $y \in \Di$ we have: $|g(y)| \leq \tanh(D/2)$. The Schwarz lemma \cite[Th. 2.3.1]{Jost} implies that for every $y \in \Di$ we have: $|g(y)| \le \tanh(D/2) \, |y|$, from which \eqref{eq:HD2} follows using that $x \mapsto \artanh(x)$ is increasing and convex on $[0,1)$.
 \end{proof}
 \begin{remark}
   \label{rem:disk}
 The $d_\H$-diameter of a complex disk $D(z, R) \subset \mathbb H$, where $z \in \mathbb H$ and $R \in [0, \Re(z)]$, is given by
 \begin{equ} \label{e:DiameterDisk}
 \mathrm{diam}_{\mathbb H}(D(z,r)) = 2 \artanh\Big(\frac{R}{\Re(z)}\Big) \, ,
\end{equ}
 so that, in view of Lemma \eqref{lem:EasyBoundContraction}, if $f(\mathbb H) \subset D(z,R)$ then it holds $\tau_{\mathbb H}(f) \leq \frac{R}{\Re(z)}$.

\end{remark}

\subsection{Action of $2\times 2$  matrices on the Riemann sphere}

Let $\hat{\mathbb C} \eqdef \C \sqcup \{\infty\}$ be the Riemann sphere, which is in natural bijection with the projective complex plane $P(\C^2)$ via the map
\begin{equ}
z \in \C \mapsto \C \begin{pmatrix}
z \\
1
\end{pmatrix} \, , \quad \infty \mapsto \C \,\begin{pmatrix}
1 \\
0
\end{pmatrix}   \, .
\end{equ}
An invertible $2 \times 2$ complex matrix $M$ naturally acts on $P(\C^2)$, and thus also on $\hat{\C}$ via the previous mapping. More precisely, $M$ acts as an homography $T_M$ on $\hat{\C}$ as follows:
\begin{equ} \label{e:Homography}
T_M \cdot z \eqdef \frac{a z + b}{c z + d} \ \mbox{ if } M= \begin{pmatrix}
a & b \\
c & d
\end{pmatrix} \, .
\end{equ}
Let us  mention a useful fact: the range of a complex disk or half-plane by an homography is also a complex disk or half-plane.

 Let us introduce also
\begin{equ}
\cH \eqdef \Big\{
\begin{pmatrix}
z_1 \\
z_2
\end{pmatrix}, \, z_1, z_2 \in \C \,:\; \Re(z_1 \overline{z_2}) > 0\Big\} \, ,
\end{equ}
and let $\cM$ be the semi-group of complex invertible $2\times 2$  matrices $M$ such that $M(\cH) \subset \cH$ and $M^*(\cH) \subset \cH$.
Note that if $M\in \cM$, then the associated homography $T_M$ maps $\H$ to itself.

For $\alpha > 0$ and $\omega \in \C$ such that $\Re(\omega) \geq 0$,  consider the $2 \times 2$ matrix
\begin{equ}
  \label{matrixT}
M(\omega, \alpha) \eqdef \begin{pmatrix}
\omega & \alpha \\
1 & 0
\end{pmatrix} \, ,
\end{equ}
which acts on $\hat{\C}$ as
\begin{equ}
T_M \cdot z = \omega + \frac{\alpha}{z} \, .
\end{equ}
Note that the matrix $M^\theta_{H_2}(w_1,w_2)$ appearing in \eqref{e:Matrix} is a special case of the above matrix, (which indeed satisfies $\Re(\omega) \geq 0$ because $\theta \in [0, \pi/2]$).
It is immediate to see that $M(\omega,\alpha)\in\cM$, so in particular $\tau_\H(T_{M(\omega,\alpha)})\le 1$.
On the other hand, the diameter of $T_{M(\omega,\alpha)}(\mathbb H)$ for the distance $d_{\mathbb H}$ is not finite (as follows from the fact that $T_M$ maps $0$ to $\infty$). In fact, one can see that $\tau_{\mathbb H}(T_M) = 1$: $T_M$  does not strictly contract the distance. However, the product of $T_1$ and $T_2$ corresponding to two matrices $M_1,M_2$ of the form \eqref{matrixT} such that $\Re(\omega_1), \Re(\omega_2) > 0$ contracts strictly  the Poincaré metric:
\begin{lemma}
  \label{lem:duedischi}
  It holds
\begin{equ}
\tau_{\mathbb H}(T_{M(\omega_1, \alpha_1)} T_{M(\omega_2, \alpha_2)}) \leq \Big(1+\frac{2 \Re(\omega_1) \Re(\omega_2)}{\alpha_1}\Big)^{-1}.
\end{equ}
\end{lemma}

\begin{proof}
Denoting $T_i = T_{M(\omega_i, \alpha_i)}$ for $i \in \{1,2\}$, we have
\begin{equs}
T_2(\mathbb H) &= \{z : \Re(z) > \Re(\omega_2)\} \, \\
T_1 (T_2(\mathbb H)) &= D\Big(\omega_1 + \frac{\alpha_1}{2 \Re(\omega_2)}, \frac{\alpha_1}{2 \Re(\omega_2)}\Big) \, ,
\end{equs}
which combined with Lemma \ref{lem:EasyBoundContraction} and Remark \ref{rem:disk}, gives
\begin{equ}
\tau_{\mathbb H}(T_1 T_2) \leq \Big(1+\frac{2 \Re(\omega_1) \Re(\omega_2)}{\alpha_1}\Big)^{-1} \, .
\end{equ}
\end{proof}
An important consequence is the following:
\begin{corollary}\label{cor:map}
  Let $(M_i)_{i\ge 1}=(M(\omega_i,\alpha_i))_{i\ge1}$ be a family of matrices of the form \eqref{matrixT}, and assume that $\inf_i \Re(\omega_i)>0$ and $\sup_i\alpha_i<\infty$. Then, there exists a unique $z_0 \in \overline{\mathbb H}$ such that, for every $z\in\mathbb H$,
  $T_{M_1}\dots T_{M_n}\cdot z\to z_0$ as $n\to\infty$. The diameter of $T_{M_1}\dots T_{M_n}(\overline{\H})$ shrinks to zero exponentially fast.
\end{corollary}
We also mention the following fact that we need later:
\begin{lemma}
  \label{lemma:trace}
  Let $A, B \in \cM$ be complex invertible $2\times 2$ matrices and assume that $T_A(\H)$, $T_{A^*}(\H)$, $T_B(\H)$, $T_{B^*}(\H)$ are contained in the closed ball $B_{d_\H}(1,R)$ for some $R < \infty$. Then, there exists $c_R > 0$ such that
  \begin{equs}
  \label{eq:Trnorm}
  c_R \|A\|_1 \le |{\rm Tr}&(A)|\le \|A\|_1 \, , \\
  \label{eq:Prodnorm}
  c_R \|A\|_1 \|B\|_1 \le  \|A & B\|_1 \le \|A\|_1 \|B\|_1 \, .
  \end{equs}
\end{lemma}
\begin{proof}
  The upper bounds are trivial and hold for any matrix.

  For the lower bound in \eqref{eq:Trnorm}, write $A$ as in \eqref{e:Homography}. Since $T_A(\H) \subset B_{d_\H}(1,R)$ and $T_{A^*}(\H)\subset B_{d_\H}(1,R)$, and since $B_{d_\H}(1,R)$ is stable by inversion, we deduce that the ratios
  \begin{equ}
  a/c = T_A \cdot \infty \, , \ b/d = T_A \cdot 0 \, , \ a/b = T_{A^*} \cdot \infty \, , \ c/d = T_{A^*} \cdot 0 \, ,
  \end{equ}
  and their inverses belong to $B_{d_\H}(1,R)$. Since $\arg(z), |z|$ are bounded away from respectively $\pm \pi/2$ and $0,\infty$ for $z\in B_{d_\H}(1,R)$, we deduce that there exists $\delta_R > 0$ such that the angle $\theta$ between $a$ and $d$ is at most $\pi-\delta_R$, and such that the ratios between any two matrix elements of $A$ are bounded between $\delta_R$ and $1/\delta_R$.
  Hence,
  \begin{equ}
    |a+d|^2=(|a|-|d|)^2+2|a| |d|(1+\cos(\theta))\ge  2|a| |d| \, (1+\cos(\theta)) \gtrsim_R |a| |d| \gtrsim_R \|A\|_1^2 \, .
  \end{equ}

For the lower bound in \eqref{eq:Prodnorm},  define $\cH_R = \Big\{
\begin{pmatrix}
z_1 \\
z_2
\end{pmatrix} \in \C^2 \,:\; z_1/z_2 \in D_{\H}(1,R)\Big\}$, and note that it is enough to show that for every $M \in \cM$ such that $M$ and $M^*$ map $\cH$ into $\cH_R$, we have
\begin{equ}
\inf_{v \in \cH_R \setminus \{0\}} \frac{\|M v\|_1}{\|v\|_1} \gtrsim_R \|M\|_1 \, .
\end{equ}
To show the latter, it is enough to show that for every $z \in B_{\H}(1,R)$ it holds
\begin{equ}
\Big\|M \begin{pmatrix}
z \\
1
\end{pmatrix} \Big\|_1 \gtrsim_R \|M\|_1 {\Big\|\begin{pmatrix}
z \\
1
\end{pmatrix} \Big\|_1}\, .
\end{equ}
On the one hand using that $D_{\H}(1,R)$ is bounded,  $\Big\|\begin{pmatrix}
z \\
1
\end{pmatrix} \Big\|_1 \lesssim_R 1$, and on the other hand, writing $M$ as in \eqref{e:Homography} we have
\begin{equ}
\Big\|M \begin{pmatrix}
z \\
1
\end{pmatrix} \Big\|_1 = \Big\|\begin{pmatrix}
a + b z \\
c + d z
\end{pmatrix} \Big\|_1 \geq |a + b z| = |a| \Big|1 + \frac{b}{a} z\Big| \, .
\end{equ}
As above, we know that $b/a$ and $z$ belongs to $D_{\H}(1,R)$, so that the angle between $1$ and $\frac{b}{a} z$ is at most $\pi - \delta_R$ and thus $|1 + \frac{b}{a} z| \gtrsim_R 1$. We conclude noting as above that the ratios between any elements of $A$ are bounded between $\delta_R$ and $1/\delta_R$, so that $|a| \gtrsim_R \|M\|_1$. This concludes.
\end{proof}

\section{General properties of $z \mapsto L(z)$}
\label{sec:Lyapunov}

Let us first recall Furstenberg's theorem \cite{Furst} for the special case of the product of i.i.d. $2 \times 2$ invertible complex matrices (we will use actually the version in \cite[Sec. II.4]{Bougerol}). 

\begin{theorem}
\label{th:Furst}
Let $\P$ be a probability measure on $\{V \in GL_2(\C) \, , \, |\det(V)| = 1\}$ satisfying  $\int \log\|V\| \, d\P(V)<\infty$ for some matrix norm $\|\cdot\|$. Let $G_\P$ be the smallest closed subgroup of $GL_2(\C)$ which contains the support of $\P$.
Assume that:
\begin{enumerate}
\item (Non-compactness) $G_\P$ is not compact.

\item (Irreducibility) $G_{\mathbb P}$ does not preserve any subsets of $P(\C^2)$ with cardinality $1$ or $2$.
\end{enumerate}
Then, there exists an unique $\mathbb P$-invariant measure, denoted $\nu$. This measure has no atoms, and the top Lyapunov exponent associated to $\P$ is strictly positive. 
 \end{theorem}
From the assumption $|\det(V)| = 1$, it is trivial to see that $\mathrm{Lyap}(V) \geq 0$, but the strict positivity is a non-trivial consequence of the irreducibility.

We equip $\C^2$ with its canonical hermitian product
$\langle \cdot, \cdot \rangle$ (linear and skew-linear respectively in
the first and second argument) and with its canonical basis
$(e_1, e_2)$. For $z \in \C$ let us define
\begin{equ} \label{e:Matrixnewz}
\M^z = \begin{pmatrix}
z w_1 & w_2^2 \\
1 & 0
\end{pmatrix} \, , \quad L(z) \eqdef \mathrm{Lyap}(\M^z) \ge0 \, ,
\end{equ}
and note that $M_{H_2}^\theta = \M^{2 \sin(\theta +\iota  H_2)}$.

The function $L$ is symmetric with respect to reflection across the  real axis, $L(\bar z)=L(z)$, and across the imaginary axis, $L(z)=L(-\bar z)$. To see the latter, note that
\begin{eqnarray}
  \label{eq:Mz-z}
\M^{-z}=-
\begin{pmatrix}
  1 & 0\\ 0& -1
\end{pmatrix} \M^z\begin{pmatrix}
  1 & 0\\ 0& -1
                 \end{pmatrix},
\end{eqnarray}
so that $L(z)=L(-z)$. Note that if $z \in \H$, then $\M^z$ is of the form \eqref{matrixT}, in particular $\M^z \in \cM$.

 Invariance by conjugation of the trace and \eqref{eq:Mz-z} imply that
\begin{eqnarray}
  \label{eq:tr-z}
  {\rm Tr}(\M_{-n+1}^z \dots \M_{n}^z) = {\rm Tr}(\M_{- n+1}^{-z} \dots \M_{n}^{-z}), \quad n\in \N,
\end{eqnarray}
where $\M_i^z$ are matrices of the form \eqref{e:Matrixnewz} with weights $(w_1(i), w_2(i))$.
\begin{lemma}
\label{lem:alternative}
For $n \geq 1$ and $z \in \mathbb H$ let us define
\begin{equ} \label{eq:alternative}
\psi_n(z) := \frac{1}{2n} \log |P_n(z)|,\quad P_{n}(z)\eqdef{\rm Tr}( \M_{-n+1}^{z} \dots \M_{n}^{z}) \, .
\end{equ}
Then $\P$-a.s., $(\psi_n)_{n \geq 0}$ converges (together to all its derivatives) to $L$, uniformly on compact subsets of $\H$.
Moreover, the roots of the polynomial $P_{n}(z)$ belong to $\iota \R$ and they come in complex conjugate pairs.
\end{lemma}

\begin{proof}
Let us first check that for every (fixed) $z \in \H$, it holds $\psi_n(z) \to L(z)$ $\mathbb P$-almost surely. To do so, note that the matrices $A_n^z = \M_{-n+1}^z \dots \M_0^z$ and $B_n^z = \M_1^z \dots \M_n^z$ (and thus also $A_n^z B_n^z$) satisfy the assumption of Lemma \ref{lemma:trace} for some finite $R$ independent of $n$ (this follows from Corollary \ref{cor:map}). Thus from Lemma \ref{lemma:trace} we deduce
\begin{equ}
|\mathrm{Tr}(A_n^z B_n^z)| \asymp_R \|A_n^z B_n^z\|_1 \asymp_R \|A_n^z\|_1 \|B_n^z\|_1 \, .
\end{equ}
By the definition of the Lyapounov exponent we have $\mathbb P$-almost surely $\frac{1}{n} \log \|A_n^z\| \to L(z)$ and $\frac{1}{n} \log \|B_n^z\| \to L(z)$, so that we also deduce $\psi_n(z) \to L(z)$.

The function $P_n$ is holomorphic  on $\H$  (it is a polynomial), does not vanish by \eqref{eq:Trnorm}, and has positive values when restricted to $\R_+^*$. Since $\exp \colon \C \to \C^*$ is an holomorphic covering and $\H$ is obviously simply connected, there exists an unique holomorphic lifting, denoted $\log P_n$, of $P_n$ by the exponential covering which coincide with the usual logarithm on $\R_+^*$. Then let $f_n := \frac{1}{n} \log P_n$, which is thus holomorphic on $\H$. Observe that for every compact subset $K$ of $\H$ and $z \in K$ we have
\begin{eqnarray}
  \label{eq:Montel}
  \Re (f_n(z))=\frac{1}{n} \log |{\rm Tr}(\M_{-n+1}^z \dots \M_n^z)| \le C_K
\end{eqnarray}
for a constant $C_K \in (0, \infty)$ independent of the disorder realization, simply bounding the trace with the product of the norms of the matrices and recalling that disorder is bounded.
As a consequence, Montel's Fundamental Normality Test \cite[Sec. 2.7]{Schiff}
together with Weierstrass Theorem \cite[Sec. 1.4]{Schiff} implies that $(f_n)_{n \geq 1}$ is relatively compact in the space of holomorphic functions equipped with the topology of uniform convergence on compact subsets of the function and all its derivatives. But since $\mathbb P$ almost surely $f_n(z) \to L(z)$ for every $z \in \Q \cap \R_+^*$, we deduce that $\mathbb P$ almost surely $(f_n)$ converges for the latter topology to a non-random holomorphic function denoted $f$. Since $\Re(f_n) = \psi_n$ and  $\Re(f) = L$, this concludes the proof.

As for the last claim: we know that   $P_{n}(z)$ does not vanish on $\H$. Because of \eqref{eq:tr-z}, $P_{n}$ cannot vanish on the open left half-plane, either, so its zeros are on the imaginary axis and if $\iota \beta$ is a zero, then $-\iota \beta$ is as well (with the same multiplicity).
\end{proof}

\begin{remark}
  With a similar (actually simpler) proof, we also have that $\frac1n\log |{\rm Tr}(\M^z_0\dots \M^z_n)|$ converges to $L(z)$ (almost surely and uniformly on compact sets of $\H$).
  \label{rem:also}
\end{remark}

The purpose of this section is to prove the following properties of $L$.

\begin{lemma} \label{lem:GeneralLyap}
The function $L$ is continuous and sub-harmonic on $\C$, is harmonic on $\C \setminus \iota \R$ and it is strictly positive on $\C\setminus\{0\}$. In particular, $L$ is real analytic on $\mathbb H$.

As $z \searrow 0$ along the real axis or $z \to 0$ along the imaginary axis, it holds
\begin{equ}
   \label{eq:Lreax}
L(z) \sim \frac{\mathrm{Var}(\log w_2)}{\log(1/|z|)} \, .
\end{equ}
Finally, $x \in \R_+ \mapsto L(x+iy)$ is strictly increasing for any $y \in \R$, and furthermore it holds $\partial_{\mathrm{Re}(z)} L > 0$ on $\H$.
\end{lemma}


\begin{proof}
\underline{Subharmonicity on $\mathbb C$.} By definition we have for every $z \in \C$
\begin{equ}
L(z) = \lim_{n \to\infty} 2^{-n} \, \E[\log \|\M_1^z \dots \M_{2^n}^z\|_\infty] \, ,
\end{equ}
and note that, by sub-additivity, this sequence is non-increasing in $n$. Using that (at fixed $\omega \in \Omega$) the functions $z \mapsto (\M^z_1 \dots \M_n^z)_{i,j}$, $i,j \in \{1,2\}$, are complex analytic (they are polynomials) while the function $w \in \C \mapsto \log |w|$ is sub-harmonic, we deduce that the function $z \mapsto \log |(\M_1^z \dots \M_n^z)_{i,j}|$ is sub-harmonic. Since sub-harmonicity is stable by taking maximum, we deduce that $z \mapsto \log \|\M_1^z \dots \M_n^z\|_{\infty}$ is sub-harmonic. Taking expectation (and using submultiplicativity of the norm and boundedness of the disorder) we deduce that $z \mapsto \frac{1}{n} \E[\log \|\M_1^z \dots \M_n^z\|_\infty]$ is sub-harmonic and uniformly bounded in $n$ on any compact subsets. This implies that their (non-increasing) point-wise limit, that is $L$, is sub-harmonic on $\C$.

\underline{Harmonicity on $\mathbb H$.}
The proof of Lemma \ref{lem:alternative} shows in particular that, on $\mathbb H$, $L = \Re(f)$ for some holomorphic function $f$. This implies that $L$ is harmonic on $\mathbb H$ (by symmetry, also on $- \mathbb H$).

\underline{Non-compactness and irreducibility on $\iota \R_*$.} Let $x \in \R_*$. To study $\M^{\iota x}$ it is convenient to notice that
\begin{equ} \label{e:Vx}
  \label{Vx}
\M^{\iota x} = \iota w_2 \begin{pmatrix}
\iota & 0 \\
0 & 1
\end{pmatrix}
V^x
\begin{pmatrix}
\iota & 0 \\
0 & 1
\end{pmatrix}^{-1} \, , \quad V^x = V^x(w_1, w_2) := \begin{pmatrix}
x w_1/w_2 & - w_2\\
1/w_2 & 0
\end{pmatrix} \, .
\end{equ}
In particular $L(\iota x) = \mathrm{Lyap}(V^x)$. The advantage of this representation is that $V^x$ belongs to $SL_2(\R)$. Let us show that the law $\P_x$ of $V^x$ satisfies the assumptions of Theorem \ref{th:Furst}.

If $G_{\P_x}$ were compact, then it would be conjugate to $SO(2)$, that is, there would exist $M \in SL_2(\R)$  such that every $g\in G_{\P_x}$ is of the form $g = M R M^{-1}$, with $R \in SO(2)$. In particular, $G_{\P_x}$ would be abelian. On the other hand, the commutator between two independent copies $V^x(1),V^x(2)$ of $V^x$ is not zero almost surely, for instance its $(1,1)$ element is proportional to $(w_2(1))^2-(w_2(2))^2$, which is not identically zero under Assumption \ref{ass:disordine}.

To check the irreducibility, note first that
the eigenvalues and eigenvectors of $V^x$ are
\begin{eqnarray}
  \label{vlambda}
\lambda_\pm=  \lambda_\pm(x,w_1,w_2)=\frac{x w_1\pm\sqrt{x^2w_1^2-4 w_2^2}}{2w_2},\qquad v_\pm=v_\pm(x,w_1,w_2)=
  \begin{pmatrix}
    w_2\lambda_\pm\\1
  \end{pmatrix}.
\end{eqnarray}
If there exists $z \in \hat{\C}$ such that $T_V \cdot z = z$ for every $V \in G_{\P_x}$, it follows that $(z \ 1)^T$ is almost surely an eigenvector of $V$. This contradicts \eqref{vlambda}, since $w_2 \lambda_\pm$ is not deterministic, under Assumption \ref{ass:disordine}. If instead there exist $z_1, z_2 \in \hat{\C}$ such that $T_V \cdot z_1, T_V \cdot z_2 \in \{z_1,z_2\}$ for every $V\in G_{\P_x}$, then it follows that $(z_1 \ 1)^T$ is almost surely an eigenvector of $V^2$. But it is easily checked that the eigenvectors of $V^2$ are also given \eqref{vlambda}, so that by the same argument as previously we deduce a contradiction.

\underline{Positivity on $\C \setminus \{0\}$.} By the non-compactness and irreducibility of $V^x$ for $x \in \R_*$, which were proved previously, and Theorem \ref{th:Furst} we deduce that $L(\iota x) > 0$ for every $x \in \R_*$. Using that, on $\H$, $L$ is non-negative (because $\E[\log |\det(\M^z)|] = 0$), harmonic, and non-constant (for instance, $L(x) \geq \E[\log(2 w_1 x)] \to +\infty$ as $x \to +\infty$), we deduce by the maximum principle that $L$ is positive on $\H$ (and also on $- \H$ by symmetry).

\underline{Continuity on $\C$.} The continuity of $L$ on $\C \setminus \iota \R$ is an obvious consequence of the harmonicity of $L$. The continuity at $0$ can be easily deduced using that $L(0) = 0$, $L \geq 0$ on $\C$ and that $L$ is sub-harmonic so in particular is upper-semi continuous. For $x \ne 0$, the continuity of $L$ at $\iota x$ follows from the non-compactness and irreducibility of $V^x$, which were proved previously, and \cite[Th. B]{FK}.

\underline{Monotonicity on $\overline\H$ with respect to $\Re(z)$.}
Let $\psi_n$ be as in \eqref{eq:alternative}.
By Lemma \ref{lem:alternative},
\begin{eqnarray}
  \label{eq:|tr|}
\psi_{n}(z) = c_n+\frac1{2n}\sum_{i=1}^n \log (|z-\iota \beta_{i,n}||z + \iota \beta_{i,n}|)
\end{eqnarray}
for $z\in\overline\H$, with $c_n$ a constant and $\beta_{i,n}\in \R$. Each summand is increasing with respect to the real part of $z$ for $\Re(z) > 0$. By taking the $n\to\infty$ limit and Lemma \ref{lem:alternative}, we deduce that $L(z)$ is increasing with respect to the real part of $z$ for for $\Re(z)>0$, and by continuity also for $\Re(z) \geq 0$.

\underline{Positivity of $\partial_{\Re(z)}L(z)$.}
Since $z \in \H \mapsto L(z)$ is harmonic, so is also $z \in \H \mapsto \partial_{\Re(z)} L(z)$. The latter is non-negative by the previous point (and non identically equal to $0$ since $L(x) \to +\infty$ as $x \to +\infty$), so that by the maximum principle for harmonic functions we deduce that $\partial_{\Re(z)} L(z)$ is strictly positive on $\H$.

\underline{Asymptotic of $L(\eps)$ as $\eps \searrow 0$.} The essential ingredient is provided by the main result of \cite{collin2024large} (or \cite[Th. 3]{de2024scaling}), which implies that, for every non-deterministic random variable $Z$ compactly supported in $(0, +\infty)$ such that $\E[\log Z] = 0$, it holds as $\eps \searrow 0$
\begin{equ} \label{e:LyapCriticalBehavior}
  {\rm Lyap}\left(
   \begin{pmatrix}
1 & \eps \\
\eps Z & Z
\end{pmatrix}\right) \sim \frac{\mathrm{Var}(\log Z)}{4 \log(1/\eps)} \, .
\end{equ}
(Under the additional assumption that the law of $\log Z$ admits a density,  \cite{collin2025lyapunov} gives the sharper asymptotic expression
\begin{eqnarray}
  \label{eq:sharper}
 {\rm Lyap}\left( \begin{pmatrix}
1 & \eps \\
\eps Z & Z
\end{pmatrix}\right) = \frac{\mathrm{Var}(\log Z)}{4 \log(1/\eps)+\kappa}+O(\epsilon^\delta)
\end{eqnarray}
for some $\kappa\in \R,\delta>0$.)
Note that
\begin{equ}
\M_{2i}^\eps \M_{2i+1}^\eps = \begin{pmatrix}
w_2(2i)^2 + \eps^2 \, w_1({2 i}) w_1({2 i + 1}) & \eps w_1({2i}) w_2({2i+1})^2 \\
\eps w_1({2i+1}) & w_2({2i+1})^2 \,
\end{pmatrix}.
\end{equ}
Since the random variables $w_1,w_2$ are bounded away from $0$ and $\infty$,  we can bound
\begin{equ}
  \label{suegiu}
w_2({2i})^2 \begin{pmatrix}
1 & C^{-1} \eps \\
C^{-1} \eps \frac{w_2({2i+1})^2}{w_2({2i})^2}& \frac{w_2({2i+1})^2}{w_2({2i})^2} \, .
\end{pmatrix} \leq \M_{2i}^\eps \M_{2i+1}^\eps \leq (1 + C \eps^2) w_2({2i})^2 \begin{pmatrix}
1 & C \eps \\
C \eps\frac{w_2({2i+1})^2}{w_2({2i})^2} & \frac{w_2({2i+1})^2}{w_2({2i})^2} \, .
\end{pmatrix} \, ,
\end{equ}
where we write $A \leq B$ to mean that $A_{i,j} \leq B_{i,j}$ for every $i,j \in \{1,2\}$. As a consequence, recalling that $\log w_2$ is centered, it holds
\begin{equ}
\frac{1}{2}   {\rm Lyap}\left( \begin{pmatrix}
1 & C^{-1} \eps \\
C^{-1} \eps Z& Z \, .
\end{pmatrix}\right) \leq L(\eps) \leq \frac{1}{2} \log(1 + C \eps^2) + \frac{1}{2} {\rm Lyap}\left( \begin{pmatrix}
1 & C \eps \\
C \eps Z& Z \,
\end{pmatrix}\right) \,
\end{equ}
where $Z=\frac{w_2({1})^2}{w_2({0})^2}$.
Because of \eqref{e:LyapCriticalBehavior}, both left and right hand sides are equivalent to $\frac{\mathrm{Var}(\log w_2)}{\log(1/\eps)}$ as $\eps \searrow 0$, which concludes.

\underline{Asymptotic of $L(\iota \eps)$ as $\eps \to 0$.} Recall that $L({\iota \eps}) = {\rm Lyap}(V^\eps)$ with $V^\eps$ defined in \eqref{Vx}; write ${\rm Lyap}(V^\eps)=\frac12{\rm Lyap}(W^\eps)$ with $W^\eps$ equal minus the product of two independent copies of $V^\eps$:
\begin{eqnarray}
  W^\eps =
  \begin{pmatrix}
    \frac{w_2}{w_2'} - \eps^2 \frac{w_1 w_1'}{w_2 w_2'} & \eps w_1 \frac{w_2'}{w_2}\\
- \eps \frac{w_1'}{w_2 w_2'} & \frac{w_2'}{w_2}
  \end{pmatrix}
\end{eqnarray}
  where $(w_1',w_2')$ is an i.i.d. copy of $(w_1,w_2)$.
  Defining the random variables
  \begin{eqnarray}
\kappa\eqdef\frac{w_2}{w_2'},\quad a\eqdef \frac12\Big(w_1+\frac{w_1'}{w_2^2}\Big),\quad
b\eqdef     \frac12\Big(-w_1+\frac{w_1'}{w_2^2}\Big),
  \end{eqnarray}
  we have
  \begin{eqnarray}
    W^\eps=\left[I- \eps a
    \begin{pmatrix}
      0& -1\\1& 0
    \end{pmatrix} - \eps b \begin{pmatrix}
      0& 1\\1& 0
                      \end{pmatrix}- \eps^2
    \begin{pmatrix}
       \frac{w_1 w_1'}{w_2 w_2'} &0\\0&0
    \end{pmatrix}
    \right]\begin{pmatrix}
      \kappa& 0\\0& \frac1\kappa.
    \end{pmatrix}
  \end{eqnarray}
  Note that $\kappa$ is compactly supported in $(0,\infty)$ and its logarithm is centered; $a,b$ are compactly supported and $a>|b|$. Then, \cite[Th. 1]{de2024scaling} implies that, as $|\eps| \to 0$,
  \begin{eqnarray}
    \label{eq:baldes}
    {\rm Lyap}(W^\eps)\sim \frac{{\rm Var}(\log \kappa)}{|\log(|\eps|)|}=2\frac{{\rm Var}(\log w_2)}{|\log(|\eps|)|}.
  \end{eqnarray}
  As a consequence, $L(\iota \eps)\sim \frac{{\rm Var}(\log w_2)}{|\log(|\eps|)|}$ as $x \to 0$.
\end{proof}

Finally, a fact that will turn out useful later:
\begin{lemma}
  \label{lem:miracolo}
  Let $a, b > 0$ and for $\theta\in[0,\pi/2]$ define $z(\theta) = \iota a e^{-\iota\theta} - \iota b e^{\iota \theta}$. Then, the function $\theta \in [0, \pi/2] \mapsto L(z(\theta))$ is strictly increasing, and satisfies
  \begin{equs}
    \label{eq:miracolo}
    \frac{\partial}{\partial \theta} L(z(\theta)) &> 0, \ \theta \in (0, \pi/2) \\
    \restriction{\frac{\partial}{\partial \theta}}{\theta = \pi/2} L(z(\theta)) = 0 \, &, \quad \restriction{\frac{\partial^2}{\partial \theta^2}}{\theta = \pi/2} L(z(\theta)) < 0 \, . \label{eq:miracolo'}
  \end{equs}
\end{lemma}
Note that $\theta:[0,\pi/2]\mapsto z(\theta)$ describes a quarter of ellipse in $\C$, with larger horizontal axis, starting at $\iota(a-b)$ and ending at $a+b$.
\begin{proof}
For $\theta \in [0, \pi]$ , we define $z(\theta) = \iota a e^{-\iota\theta} - \iota b e^{\iota \theta}$. Since $L$ is real analytic on $\H$, we deduce that $\theta \in (0, \pi) \mapsto L(z(\theta))$ is also real analytic. Furthermore, since $L$ is symmetric by complex conjugation, we have $L(z(\theta)) = L(z(\pi - \theta))$, which in particular implies the first point of \eqref{eq:miracolo'}.

We first claim that if $a, b \geq 0$, then $\theta \mapsto L(z(\theta))$ is non-decreasing on on $[0, \pi/2]$. We start from \eqref{eq:|tr|}, with $z=z(\theta)$.
An elementary computation shows
  \begin{eqnarray}
    \partial_\theta \log (|z(\theta)-\iota \beta||z(\theta)+\iota \beta|) =
    2 \sin(2\theta) \frac{2|z(\theta)|^2a b+\beta^2(a^2+b^2)}{|z(\theta)+\beta^2|^2}\ge0 \, ,
  \end{eqnarray}
  hence $\partial_\theta \psi_n(z(\theta))\ge0$. Since the l.h.s. of \eqref{eq:|tr|} (together with all its derivatives) converges uniformly to $L$ on compact subsets of $\H$, we deduce the desired claim (using also the continuity of $L$ on $\overline \H$). In particular, this implies \eqref{eq:miracolo} and \eqref{eq:miracolo'} but with weak instead of strict inequalities.

  Let us show that these inequalities are strict as soon as $a,b > 0$. Up to switching $a$ and $b$ and using the symmetry of $L$ by complex conjugation, we can always assume that $a \geq b$. Observe that
  \begin{equs}
  \frac{\partial}{\partial \theta} &L(z(\theta)) = (a+b) \cos(\theta) \restriction{\frac{\partial L}{\partial \mathrm{Re}}}{z(\theta)} - (a-b) \sin(\theta) \restriction{\frac{\partial L}{\partial \mathrm{Im}}}{z(\theta)} \\
  &= \frac{4 a b}{a + b} \cos(\theta) \restriction{\frac{\partial L}{\partial \mathrm{Re}}}{z(\theta)} + \frac{a - b}{a + b} \Big((a-b) \cos(\theta) \restriction{\frac{\partial L}{\partial \mathrm{Re}}}{z(\theta)} - (a+b) \sin(\theta) \restriction{\frac{\partial L}{\partial \mathrm{Im}}}{z(\theta)}\Big) \\
  &= \frac{4 a b}{a + b} \cos(\theta) \restriction{\frac{\partial L}{\partial \mathrm{Re}}}{z(\theta)} + \frac{a - b}{a + b} \, \restriction{\frac{\partial}{\partial \phi}}{\pi/2 - \arg(z(\theta))} L(\iota |z(\theta)| e^{- \iota \phi}) \, .
  \end{equs}
Note that the first derivative is strictly positive by Lemma \ref{lem:GeneralLyap}, while the second derivative is non-negative by the previous claim, from which we obtain \eqref{eq:miracolo}. As for \eqref{eq:miracolo'}, observe that
\begin{equs}
\restriction{\frac{\partial^2}{\partial \theta^2}}{\theta = \pi/2} &L(z(\theta)) = - (a+b) \restriction{\frac{\partial L}{\partial \mathrm{Re}}}{a+b} + (a-b)^2 \restriction{\frac{\partial^2 L}{\partial \mathrm{Im}^2}}{a+b} \\
&= - \frac{4 a b}{a+b} \restriction{\frac{\partial L}{\partial \mathrm{Re}}}{a+b} + \frac{(a-b)^2}{(a+b)^2} \Big(- (a+b) \restriction{\frac{\partial L}{\partial \mathrm{Re}}}{a+b} + (a+b)^2 \restriction{\frac{\partial^2 L}{\partial \mathrm{Im}^2}}{a+b}\Big) \\
&= - \frac{4 a b}{a+b} \restriction{\frac{\partial L}{\partial \mathrm{Re}}}{a+b} + \frac{(a-b)^2}{(a+b)^2} \restriction{\frac{\partial^2}{\partial \phi^2}}{\phi = \pi/2} L(\iota (a+b) e^{- \iota \phi}) \, .
\end{equs}
The first derivative is strictly positive by Lemma \ref{lem:GeneralLyap}, while the second derivative is non-positive by the previous claim, from which we obtain \eqref{eq:miracolo'}.
\end{proof}

\begin{remark} \label{rmk:Lgamma}
To deal with the case $\gamma > 1$, we will be interested in the function
\begin{equ}
  \label{eq:Lgammaz}
L_\gamma(z) := \frac{1}{2} \, \mathrm{Lyap}\Big(
\begin{pmatrix}
z w_1 & (\gamma w_2)^2 \\
1 & 0
\end{pmatrix}
\begin{pmatrix}
z w_1' &  (w_2'/\gamma)^2 \\
1 & 0
\end{pmatrix}\Big) \, ,
\end{equ}
where $(w_1', w_2')$ are independant copies of $(w_1, w_2)$. Let us highlight that Lemma \ref{lem:alternative}, Lemma \ref{lem:GeneralLyap} (without the asymptotic results) and Lemma \ref{lem:miracolo} can be easily generalized to this case. The main difference is that $L_\gamma(0) = \log \gamma > 0$ (and let us mention that the asymptotic of $L_\gamma$ at $0$ would be completely different, see for instance \eqref{e:AsympLgamma3/2} and \eqref{e:AsympLgammasingular} for more details). Actually, because of this latter fact, the continuity of $L$ at $0$ would require some extra work to prove by hand, or could be easily deduced from general theorems such as \cite[Th. B]{Avila}. Yet, for our purpose, we only need the continuity of $\restriction{L}{\R_+}$ at $0$, which can be directly checked noticing on the one hand that $L$ is upper semicontinuous (because subharmonic) and on the other hand that $\restriction{L}{\R_+}$ is non-decreasing because the entries of the involved matrices are non-negative and non-increasing in $z \in \R_+$.
\end{remark}

\section{Kasteleyn theory}

\subsection{Formula for $Z_{L,N}$ and for $F$}

The Kasteleyn matrix \cite{Kasteleyn,Gorin} $K$ of $T_{L, N}$ is a $(2L N)\times(2L N)$ matrix, whose rows (resp. columns) are indexed by the white (resp. black) vertices w/b of the graph, and it is a weighted and signed variant of the graph's adjacency matrix.
Since our graph is not planar but rather is embedded in the torus, we actually need four Kasteleyn matrices $K_{\tau},\tau=(\tau_1,\tau_2)\in\{0,1\}^2$: given an edge $(\rb,\rw)=e$, we define
\begin{equation}
  K_\tau(\rw,\rb )=
  \begin{cases}
    w_1(y) e^{-H_2} \,(-1)^{\tau_1 1_{x=1}} & \text{if }  (x',y')=(x-1,y)\\
    w_1(y) e^{H_2}\,(-1)^{\tau_1 1_{x=2L}}& \text{if } (x',y')=(x+1,y)\\
    \iota w_2(y-1) e^{H_1} \, (-1)^{\tau_2 1_{y=-N+1}} & \text{if } (x',y')=(x,y-1)\\
\iota w_2(y) e^{-H_1} \, (-1)^{\tau_2 1_{y=N}} & \text{if } (x',y')=(x,y+1)
  \end{cases}
\end{equation}
if $\rw$ is the white vertex of coordinates $(x,y)$ and $\rb$ is the black vertex of coordinates $(x',y')$. If $\rb,\rw$ are not neighbors, then $K_\tau(\rw,\rb)=0$.
We view the Kasteleyn matrices as mapping vectors labeled by black vertices to  vectors labeled by white vertices.
Kasteleyn's theorem states that
\begin{equation}
  \label{eq:Ktheorem}
    Z_{L,N}=\frac12\sum_{\tau\in\{0,1\}^2}c(\tau)\det(K_{\tau})
\end{equation}
where $c(\tau)\in\{-1,1\}$ and $\frac12|\sum_\tau c(\tau)|=1$: three of the signs are equal and the fourth is opposite.

Consider the functions $f_{k,y_0} \colon T_{L, N} \to \C$, for $k \in \mathbb R$ and $y_0 \in \{-N+1, \dots, N\}$, given by
\begin{equ} \label{e:SemiEigenfunctions}
f_{k, y_0} (x, y) \eqdef \frac{1}{\sqrt{L}} \exp\Big(\frac{\iota \pi k}{L} x\Big) \1_{y = y_0} \, ,
\end{equ}
and let us denote $f_{k, y_0}^{\mathrm{b}}$ and $f_{k, y_0}^{\mathrm{w}}$ the restriction of $f_{k, y_0}$ to respectively $V_B, V_W$. Let $F_{\tau_1}:=\{-\frac{L}{2}+1,\dots,\frac{L}{2}\}-\frac{\tau_1}2$ (recall that we are assuming $L$ to be even). Note that the family $(f_{k, y_0}^{\mathrm{b}})_{k\in F_{\tau_1}, y_0\in \{-N+1,\dots,N\}}$ (resp. $(f_{k, y_0}^{\mathrm{w}})_{k\in F_{\tau_1}, y_0\in \{-N+1,\dots,N\}}$)
is an orthonormal basis of $\C^{V_B}$ (resp. $\C^{V_W}$) for the canonical Hermitian product. It can be readily checked that for $k \in F_{\tau_1}$ and $y_0 \in \{-N+1, \dots, N\}$
\begin{multline} \label{e:ActionKasteleyn}
  K_{\tau} f^{\mathrm{b}}_{k, {y_0}} =
  w_1(y_0)\Big(
e^{-\iota\theta+H_2 }+e^{\iota\theta-H_2 }
  \Big)
  f^{\mathrm{w}}_{k, {y_0}} \\+ \iota w_2({y_0})e^{H_1} f^{\mathrm{w}}_{k, {y_0+1}}(-1)^{\tau_2  1_{ y_0=N}} + \iota  e^{\frac{2 \iota \pi}{L}} w_2({y_0}-1)e^{-{H_1}} f^{\mathrm{w}}_{k, {y_0}-1}(-1)^{\tau_2 1_{ y_0=-N+1}} \, .
\end{multline}

Thus, viewed in this basis, the matrix $K_{\tau}$ is given by a block-diagonal matrix
\begin{equ} \label{e:MatrixKasteleyn}
\mathrm{Mat}_{\mathfrak{b}}(K_{\tau}) =
\begin{pmatrix}
K^{\tau_2}_{(\tau_1/2)} & & & \\
& K^{\tau_2}_{1+(\tau_1/2)} & & \\
& & \ddots & \\
& & & K^{\tau_2}_{L-1+(\tau_1/2)}
\end{pmatrix}
\end{equ}
where for $k \in F_{\tau_1}$ we define $\theta = \frac{k \pi}{L}$ and
\begin{equ}
K^{\tau_2}_k \eqdef
\begin{pmatrix}
z_\theta(-N+1) & \iota w_2(-N+1)e^{-H_1 }  & & & (-1)^{\tau_2} \iota w_2(N) e^{H_1} \\
\iota w_2(-N+1)e^{H_1}  & z_\theta(-N+2) & \iota w_2(-N+2)e^{-H_1 } &  & \\
& \raise.2cm\hbox{$\iota w_2(-N+2)e^{H_1} $} & \ddots & \ddots \\
&  & \ddots & \ddots & \iota w_2(N-1) e^{-H_1 }  \\
(-1)^{\tau_2} \iota w_2(N) e^{-H_1} & & & \iota w_2(N-1) e^{H_1}  & z_\theta(N)
\end{pmatrix} \, .
\end{equ}
with $z_\theta(y) := w_1({y})(e^{-\iota\theta+H_2}+e^{\iota\theta-H_2})=2w_1(y)\cos(\theta +\iota H_2)$.

Let $\tilde K_k$ be the matrix defined like $K^{\tau_2}_k$, except that
$w_2(N)=0$. Observe that $\tilde K_k$ is tridiagonal and, letting $D_k(y,N)$ the determinant of $\tilde K_k$ with the first $y + N - 1$ rows and columns removed, one has
the recurrence
\begin{eqnarray}
  \label{eq:rec}
  D_k(y,N)=z_{k\pi/L}(y)D_k(y+1,N)+ w_2(y)^2D_k(y+2,N)
\end{eqnarray}
with the convention $D_k(N+1,N)=1,D_k(N+2,N)=0$.
This can be rewritten as
\begin{equs} \label{e:RecurrenceFormulaMatrix}
&\begin{pmatrix}
D_k(y,N) \\
D_k(y+1,N)
\end{pmatrix}
= R^{\frac{\pi k}{L}}_y\cdot
\begin{pmatrix}
D_k(y+1,N) \\
D_k(y+2,N)
\end{pmatrix} \,
\end{equs}
with $R^\theta_{H_2,y} = M^{\pi/2-\theta}_{-H_2,y}$ and $M^\theta_{H_2,y}$ as in Proposition \ref{p:FreeEnergyFormula}. Note that this expression is independent of $H_1$.
Thus, $\det(\tilde K_k) = D_k(-N+1,N)$ is given by the top-left entry of the matrix $R^{\frac{\pi k}{L}}_{H_2,-N+1} \dots R^{\frac{\pi k}{L}}_{H_2,N}$, which is also
\begin{equ} \label{e:DetKkFormula}
\det(\tilde K_k) =  \langle R^{\frac{\pi k}{L}}_{H_2,-N+1} \dots R^{\frac{\pi k}{L}}_{H_2,N} e_1, e_1 \rangle \, .
\end{equ}

Given a square matrix $A$ whose coefficients are naturally indexed by $\{-N+1, \dots, N\}^2$ and given indices $i_1 < \dots < i_k$, $j_1 < \dots < j_k$, let us denote $A_{\{ \setminus i_1, \dots, i_k \setminus j_1, \dots, j_k\}}$ the square matrix obtained from $A$ by removing the rows $i_1, \dots, i_k$ and the columns $j_1, \dots j_k$. Note that by multilinearity of the determinant we can write
\begin{equs}
\det(K_k^{\tau_2}) &= \det(\tilde K_k) + w_2(N)^2 \, \det((\tilde K_k)_{\{ \setminus -N+1,N \setminus -N+1, N\}}) \\
&- (-1)^{\tau_2} \iota w_2(N) e^{H_1} \det((\tilde K_k)_{\{ \setminus -N+1 \setminus N\}}) - (-1)^{\tau_2} \iota w_2(N) e^{-H_1} \det((\tilde K_k)_{\{ \setminus N \setminus -N+1\}}) \, .
\end{equs}
The first two determinants are respectively given by $D_k(-N+1, N)$ and $D_k(-N+2, N-1)$, while the two last determinants can be easily computed since they correspond to matrices which are respectively upper and lower triangular. We thus deduce (note that $\iota^{2N} = -1$ since $N$ is odd)
\begin{equs}
  \det(K^{\tau_2}_k) &= \langle R^{\frac{\pi k}{L}}_{H_2,-N+1} \dots R^{\frac{\pi k}{L}}_{H_2,N} e_1, e_1 \rangle+w_2({N})^2\langle R^{\frac{\pi k}{L}}_{H_2,-N+2} \dots R^{\frac{\pi k}{L}}_{H_2,N-1} e_1, e_1 \rangle \\
  &\qquad + (-1)^{\tau_2} 2 \cosh(
2N H_1) \prod_{y=-N+1}^{N}w_2(y) \\
 &={\rm Tr}(R^{\frac{\pi k}{L}}_{H_2,-N+1} \dots R^{\frac{\pi k}{L}}_{H_2,N}) + (-1)^{\tau_2} 2 \cosh(
2NH_1 ) \prod_{y=-N+1}^{N}w_2(y) \, . \label{detKtauk}
\end{equs}
Altogether, we have proved:
\begin{lemma} \label{lem:ExactFormulaFinite} Let $R^\theta_{H_2,y}\eqdef M^{\pi/2-\theta}_{-H_2,y}$ with $M$ as in Proposition \ref{p:FreeEnergyFormula}.
The partition function on the torus $T_{L,N}$ is given by \eqref{eq:Ktheorem}, with
\begin{eqnarray} \label{e:ExactFormulaFinite}
  \det(K_\tau)=\prod_{k\in F_{\tau_1}}\left[
  {\rm Tr}(R^{\frac{\pi k}{L}}_{H_2,-N+1} \dots R^{\frac{\pi k}{L}}_{H_2,N})+2(-1)^{\tau_2}\cosh(
2NH_1 ) \prod_{y=-N+1}^{N}w_2(y)
  \right]
.
\end{eqnarray}
\end{lemma}

Recall that, by symmetry, we are restricting to $H_1,H_2\ge0$. Next, we have:

    \begin{lemma}
      \label{lem:illimite}
      The limit $\lim_{N\to\infty}\frac1{2N}\log Z_{L,N} $ exists $\P$-a.s. and in $\L^1$, is a.s. constant and equals
      \begin{eqnarray}\label{eq:illimite}
        \Phi_{L}:=\lim_{N\to\infty}\frac1{2N}\log Z_{L,N}=
        \max_{\tau_1=0,1} \sum_{k\in F_{\tau_1}}\max\left({\rm Lyap}(R^{\frac{\pi k}{L}}_{H_2}),H_1 \right).
      \end{eqnarray}
       Moreover,
       \begin{eqnarray}
         \lim_{L\to\infty}\frac1{2L}\Phi_L=\frac1{\pi}\int_0^{\pi/2}\max(\cL(\theta,H_2 ),H_1 )\dd\theta.
         \label{illimite}
  \end{eqnarray}
    \end{lemma}

\begin{proof}[of Lemma \ref{lem:illimite}]
By Kasteleyn's formula \eqref{eq:Ktheorem},
    \begin{eqnarray}
      \label{eq:factor2}
    \max_{\tau\in\{0,1\}^2}|\det K_\tau|\le Z_{L,N}\le 2 \max_{\tau\in\{0,1\}^2}|\det K_\tau|  \, ,
    \end{eqnarray}
 where $\det K_\tau$ is given by \eqref{detKtauk}.
 We have by
Lemma \ref{lem:alternative}
  \begin{eqnarray}
 \lim_{N\to\infty}  \frac1{2N}\log | {\rm Tr}(R^{\theta}_{H_2,-N+1} \dots R^{\theta}_{H_2,N})|=
    {\rm Lyap}(R^{\theta}_{H_2}) \, ,
  \end{eqnarray}
  and by the law of large numbers (note that $\mathbb E(\log w_2)=0$ by Assumption \ref{ass:disordine}, and  that $H_1 \geq 0$)
  \begin{eqnarray}
 \lim_{N\to\infty}    \frac1{2N}\log\left(\cosh(
2NH_1 ) \prod_{y=-N+1}^{N}w_2(y)   \right) = H_1 \, .
  \end{eqnarray}

Since $\frac{1}{2N} \log Z_{L,N}$ as well as any of its limit points is convex in $H_1$, to prove the convergence \eqref{eq:illimite} it is enough to prove the latter for a dense subset of $H_1$, so that we can assume that $H_1 \ne {\rm Lyap}(R^{\frac{\pi k}{L}}_{H_2})$ for every $k \in F_{0}\cup F_1$. Then,
$
\log(    \max_\tau|\det K_\tau|)=\max_\tau\log |\det K_\tau|
$
  and
  \begin{eqnarray}
    \lim_{N\to\infty}\frac1{2N}\log |\det K_\tau|    =\sum_{k\in F_{\tau_1}}\max\left( {\rm Lyap}(R^{\frac{\pi k}{L}}_{H_2} ),H_1 \right)
  \end{eqnarray}
  where the assumption $H_1 \ne {\rm Lyap}(R^{\frac{\pi k}{L}}_{H_2} )$ rules out the possibility that ${\rm Tr}(\dots)$ combines with $\cosh(2NH_1 )\prod w_2(y)$  to produce a quantity that grows exponentially less than both terms individually. We conclude then that $\frac{1}{2N} \log Z_{N,L}$ converges $\mathbb P$-almost surely to the right hand side of \eqref{eq:illimite}. The convergence occurs also in $\mathbb L^1$ by the dominated convergence theorem, since by boundedness of the disorder $\frac{1}{2N} |\log Z_{N,L}|$ is uniformly bounded by some deterministic constant.

  Recall that $R^\theta_{H_2}=M^{\pi/2-\theta}_{-H_2}$; then, ${\rm Lyap}(R_{H_2}^{\theta})$ is even in $H_2$ and  $\theta$, and for $\theta\in[0,\pi/2]$ it equals $\cL(\pi/2-\theta,H_2)$.
  Taking the limit $L\to\infty$ of $\Phi_L/(2L)$ and recalling from Lemma \ref{lem:GeneralLyap} that $\cL(\theta,H_2) = L(2 \sin(\theta + i H_2))$ is continuous in $\theta$, we get the claim of Lemma \ref{lem:illimite}.
\end{proof}

Proposition \ref{p:FreeEnergyFormula} is now an immediate consequence of Lemmas \ref{lem:ExactFormulaFinite} and \ref{lem:illimite}.

\begin{remark}
  \label{rem:Fgamma}
  A direct inspection of the proof of Lemma \ref{lem:illimite} shows that the same procedure can be followed in the case of $\gamma\ne 1$. All what changes is that in the matrices $M^\theta_y$, the random variable $w_2(y)$ is multiplied by $\gamma$ if $y$ is even and by $1/\gamma$ if $y$ is odd. Since the product of two such matrices for neighbouring $y,y+1$ gives the matrix $T^{\theta }_\gamma$ defined in \eqref{Mgamma}, formula \eqref{eq:FgammaFormula} readily follows. The factor $1/2$ in the definition of  $\cL_\gamma$ is due to the fact that a product of $2N$ matrices of type $M^\theta_y$ turns into a product of $N$ matrices of type $T^\theta_y$.
\end{remark}

\subsection{Formula for the $2$-points correlations}
\label{sec:2pc}
Consider two
horizontal edges $e =(\rb_1,\rw_1)$ and $e' =(\rb_2,\rw_2)$ of
$T_{L, N}$ where $\rb_i$ are black vertices and $\rw_i$ are white
vertices. Assume that $e$ has the black vertex on the left and $e'$
on the right. In terms of vertex coordinates, let us put
$\rb_1=(x_1,y_1),\rw_1=(x_1+1,y_1),\rw_2=(x_2,y_2),\rb_2=(x_2+1,y_2).$ We are
interested in the two-point correlation
\begin{equ} \label{e:Cov}
\mathrm{Cov}_{L, N, \underline{w}}(e,e') = P_{L, N}(e,e' \in D) - P_{L, N}(e \in D) P_{L, N}(e' \in D) \,
\end{equ}
for $|y-y'|\gg1$.
It turns out that, as the free energy, such quantity can be exactly computed in terms of the product of the matrices $M^\theta_y$.

We use the same notations  as  in the proof of Lemma \ref{lem:ExactFormulaFinite}.
To lighten the formulas that follow, we introduce the following notation: given
 $f:\{0,1\}^2 \to \mathbb C$, we let
 \begin{equ}
   \Langle f_\tau\Rangle:=\frac{\frac12\sum_{\tau\in \{0,1\}^2}c(\tau)\det(K_\tau)f(\tau)}
   {Z_{L,N}}
 \end{equ}
 and note that $\Langle 1\Rangle=1$.
 From Kasteleyn's theory, we know that
\begin{equ} \label{e:CovKasteleynTheory}
  \mathrm{Cov}_{L, N,\underline w}(e,e')= w_1(y_1)w_1(y_2)\left[\big\Langle \det\left(K^{-1}_\tau(\rb_i,\rw_j)\right)_{1\le i,j\le 2}\big\Rangle - \big\Langle K^{-1}_\tau(\rb_1,\rw_1) \big\Rangle \, \big\Langle K^{-1}_\tau(\rb_2,\rw_2) \big\Rangle\right]
\end{equ}
so that the problem reduces to computing
\begin{equ} \label{e:InverseKasteleyn}
K^{-1}_\tau\big(\rb,\rw\big) = \big\langle K^{-1}_\tau \1_{(x',y')}, \1_{(x,y)} \big\rangle \, ,
\end{equ}
for any pair of black and white vertices of coordinates $(x,y)$ and $(x',y')$ respectively.
\begin{remark}
   If we worked on a planar graph rather than the torus, there would be no sum over $\tau$ in \eqref{e:CovKasteleynTheory}, and the product with the minus sign would cancel exactly the diagonal term in the $2\times 2$ determinant.
   In our case, the cancellation occurs only in the thermodynamic limit, see the proof of Proposition \ref{prop:formuloneCov}.
\end{remark}

Since $(f_{k,y_0}^{\mathrm{b}})_{k\in F_{\tau_1},y_0\in\{-N+1,\dots,N\}}$ (resp. $(f_{k,y_0}^{\mathrm{w}})_{k\in F_{\tau_1},y_0\in\{-N+1,\dots,N\}}$) is an orthonormal basis of $\mathbb C^{V_B}$ (resp. $\mathbb C^{V_W}$), it holds
\begin{equ} \label{e:DiracFourier}
  \1_{(x,y)} = \frac{1}{\sqrt{L}} \sum_{k\in F_{\tau_1}} \exp\Big(-\frac{\iota \pi k}{L} x\Big) f^{\mathrm{b}}_{k,y} \, , \quad \1_{(x',y')} = \frac{1}{\sqrt{L}} \sum_{k\in F_{\tau_1}} \exp\Big(-\frac{\iota \pi k}{L} x'\Big) f^{\mathrm{w}}_{k,y'} \, ,
\end{equ}
so that
\begin{equs}
\big\langle K_\tau^{-1} \1_{(x',y')}, \1_{(x,y)} \big\rangle &= \frac{1}{L} \sum_{k\in F_{\tau_1}} \exp\Big(\frac{\iota \pi k}{L} (x-x')\Big) \, \big\langle K_\tau^{-1} f^{\mathrm{w}}_{k,y'}, f^{\mathrm{b}}_{k,y} \big\rangle \, , \\
&= \frac{1}{L} \sum_{k\in F_{\tau_1}} \exp\Big(\frac{\iota \pi k}{L} (x-x')\Big) \, \Big(K^{\tau_2}_{\frac{\pi k}{L}}\Big)^{-1}[y,y'] \, ,  \label{e:InverseKasteleynFourier}
\end{equs}
where we used in the first line that $K_\tau$ "preserves the Fourier mode" so that each $\big\langle K_\tau^{-1} f^{\mathrm{w}}_{k,y}, f^{\mathrm{b}}_{\ell,y'} \big\rangle$ with $k \ne \ell$ vanishes, and in the second line we used \eqref{e:MatrixKasteleyn}.

Now, we compute $\big(K^{\tau_2}_\theta\big)^{-1}[y,y']$ for any $-N+1 < y,y' < N$ and $\theta \in \mathbb R$. We use the cofactor identity to rewrite
\begin{equ} \label{e:CofactorIdentity}
\big(K^{\tau_2}_\theta\big)^{-1}[y,y'] = (-1)^{y+y'} \, \frac{\det((K^{\tau_2}_\theta)_{\{ \setminus y' \setminus y\}})}{\det K^{\tau_2}_\theta} \, .
\end{equ}
We have already computed $\det K^{\tau_2}_\theta$ (see \eqref{detKtauk})
so that it only remains to compute $\det((K^{\tau_2}_\theta)_{\{ \setminus y' \setminus y\}})$. We first deal with the case $y < y'$. Let us recall that $\tilde K_\theta$ denotes the matrix defined like $K_\theta^{\tau_2}$, except that $w_2(N) = 0$. Using multilinearity of the determinant, it can be readily checked that
\begin{equs}
\det((K^{\tau_2}_\theta)_{\{ \setminus y' \setminus y\}}) &= \det((\tilde K_\theta)_{\{ \setminus y' \setminus y\}}) + w_2(N)^2 \det((\tilde K_\theta)_{\{ \setminus -N+1, y', N \setminus -N+1, y, N\}}) \\
&\quad + (-1)^{\tau_2} \iota w_2(N) e^{H_1} \det((\tilde K_\theta)_{\{ \setminus -N+1, y' \setminus y, N\}}) \\
&\quad + (-1)^{\tau_2} \iota w_2(N) e^{-H_1} \det((\tilde K_\theta)_{\{ \setminus y', N \setminus -N+1, y\}}) \, . \label{e:CofactorTildeK}
\end{equs}
Now observe that these matrices are of the following form:
\begin{equs} \label{e:KasteleynMinusRowColumn1}
(\tilde K_\theta)_{\{ \setminus y', \setminus y\}} &=
\begin{pmatrix}
K_{-N+1, y-1} & 0 & 0 \\
C^{(1)} & L_{y, y'-1} & 0 \\
0 & C^{(2)} & K_{y'+1, N}
\end{pmatrix} \, , \\
(\tilde K_\theta)_{\{ \setminus -N+1, y', N \setminus -N+1, y, N\}} &= \begin{pmatrix}
K_{-N+2, y-1} & 0 & 0 \\
C^{(3)} & L_{y, y'-1} & 0 \\
0 & C^{(4)} & K_{y'+1, N-1}
\end{pmatrix} \, , \\
(\tilde K_\theta)_{\{ \setminus -N+1, y' \setminus y, N\}} &=
\begin{pmatrix}
U_{-N+1, y-1} & C^{(5)} & 0 \\
0 & K_{y+1,y'-1} & C^{(6)} \\
0 & 0 & U_{y', N-1} \, ,
\end{pmatrix}
\end{equs}
where we have denoted for $i \leq j$
\begin{equs}
K_{i,j} &= \begin{pmatrix}
z_\theta(i) & \iota w_2(i)e^{-H_1} & & \\
\raise.2cm\hbox{$\iota w_2(i) e^{H_1} $} & \ddots & \ddots & \\
& \ddots & z_\theta(j-1) & \iota w_2(j-1)e^{-H_1 } \\
& & \iota w_2(j-1)e^{H_1}  & z_\theta(j)
\end{pmatrix} \, , \\
L_{i,j} &= \begin{pmatrix}
\iota w_2(i)e^{-H_1} & & \\
z_{\theta}(i+1) & \iota w_2(i+1)e^{-H_1 } &  \\
\iota w_2(i+1)e^{H_1}  & z_{\theta}(i+2) & \iota w_2(i+2)e^{-H_1 } & \\
& \ddots & \ddots & \ddots \\
& & \iota w_2(j-1)e^{H_1}  & z_\theta(j) & \iota w_2(j)e^{-H_1}
\end{pmatrix} \, , \\
U_{i,j} &= \begin{pmatrix}
\iota w_2(i) e^{H_1} & z_\theta(i+1) & \iota w_2(i+1) e^{-H_1}  \\
& \iota w_2(i+1) e^{H_1} & z_\theta(i+1) &\ddots \\
& & \iota w_2(i+2)e^{H_1} & \ddots& \iota w_2(j-1) e^{- H_1} \\
& & & \ddots & z_\theta(j)\\
& & &  & \iota w_2(j) e^{H_1}
\end{pmatrix} \, ,
\end{equs}
and $C^{(1)}, \dots, C^{(6)}$ have exactly one non-zero coefficient, in a corner. Finally, the last matrix appearing in \eqref{e:CofactorTildeK} is a lower triangular matrix, but whose diagonal contains an interval of $0$, so that its determinant is zero. Since the three other matrices are block-triangular, their determinant is equal to the product of the determinant of the diagonal blocks and the matrices $C^{(1)}, \dots, C^{(6)}$ are irrelevant. The determinant of $K_{i,j}$ was computed in the previous section (it corresponds to $D_k(i,j)$) and  it is equal to
\begin{equ}
\langle R_i^{\theta} \dots R_j^{\theta} e_1, e_1 \rangle \, \quad R_i^\theta\eqdef R^\theta_{H_2=0,i},
\end{equ}
while as for $L_{i,j}$ and $U_{i,j}$, since they are triangular, their determinant is equal to the product of their diagonal coefficients.
Altogether we obtain (using again $\iota^{2N}=-1$)
\begin{equs}\label{e:InverseKasteleynFormula1}
\det((K_\theta^{\tau_2})_{\{\setminus y' \setminus y\}} )\\& = \iota^{y'-y}\langle R_{-N+1}^{\theta} \dots R_{y-1}^{\theta} e_1, e_1 \rangle  \, \langle R_{y'+1}^{\theta} \dots R_{N}^{\theta} e_1, e_1 \rangle \prod_{r=y}^{y'-1}(e^{-  H_1} w_2(r))\\
&+\iota^{y'-y} w_2(N)^2 \langle R_{-N+2}^{\theta} \dots R_{y-1}^{\theta} e_1, e_1 \rangle \,  \langle R_{y'+1}^{\theta} \dots R_{N-1}^{\theta} e_1, e_1 \rangle \prod_{r=y}^{y'-1}(e^{-  H_1} w_2(r))\\
&+ (-1)^{\tau_2} \iota^{2N - (y' - y)} \langle R_{y+1}^{\theta} \dots R_{y'-1}^{\theta} e_1, e_1 \rangle \prod_{r\in \{-N+1,\dots,y-1\}\cup\{y',\dots,N\}}(e^{H_1}w_2(r))  \, \\
&= \iota^{y'-y}\prod_{r=y}^{y'-1}(e^{-  H_1} w_2(r)) \, \sum_{a=1,2} \langle R_{-N+1}^{\theta} \dots R_{y-1}^{\theta} e_1, e_a \rangle \, \langle R_{y'+1}^{\theta} \dots R_{N}^{\theta} e_a, e_1 \rangle  \\
&- (-1)^{\tau_2} \iota^{- (y' - y)}\langle R_{y+1}^{\theta} \dots R_{y'-1}^{\theta} e_1, e_1 \rangle
\prod_{r\in \{-N+1,\dots,y-1\}\cup\{y',\dots,N\}}(e^{H_1}w_2(r)) \, .
\end{equs}

Now, if $y' < y$, then notice that applying the transpose, we have
\begin{equ}
\det((K_\theta^{\tau_2})_{\{ \setminus y' \setminus y\}}) = \det(((K_\theta^{\tau_2})_{\{ \setminus y' \setminus y\}})^T) = \det((K_\theta^{\tau_2})^T_{\{ \setminus y \setminus y'\}}) \, .
\end{equ}
Observe that $(K_\theta^{\tau_2})^T$ is exactly the matrix $K_\theta^{\tau_2}$ but with $H_1$ changed to $-H_1$. Thus, the previous computations implies
\begin{equs} \label{e:InverseKasteleynFormula2}
&\det((K_\theta^{\tau_2})_{ \setminus y', \setminus y\}}) \\
&= \iota^{y-y'} \prod_{r=y'}^{y-1} (e^{H_1} w_2(r)) \sum_{a=1,2} \langle R_{-N+1}^{\theta} \dots R_{y'-1}^{\theta} e_1, e_a \rangle \, \langle R_{y+1}^{\theta} \dots R_{N}^{\theta} e_a, e_1 \rangle \\
&- (-1)^{\tau_2} \iota^{- (y - y')}
\langle R_{y+1}^{\theta} \dots R_{y'-1}^{\theta} e_1, e_1 \rangle    \prod_{r\in \{-N+1,\dots,y'-1\}\cup\{y,\dots,N\}}(e^{-H_1}w_2(r)) \, .
\end{equs}
Finally, if $y = y'$, it can be readily checked that
\begin{equation} \label{e:InverseKasteleynFormula3}
\det((K^{\tau_2}_\theta)_{\{ \setminus y \setminus y\}}) =   \sum_{a=1,2}
   {\langle R_{-N+1}^{\theta} \dots R_{y-1}^{\theta} e_1, e_a \rangle \, \langle R_{y+1}^{\theta} \dots R_{N}^{\theta} e_a, e_1 \rangle}\,.
\end{equation}

Formulas \eqref{e:CovKasteleynTheory}--\eqref{e:CofactorIdentity}, together with \eqref{e:InverseKasteleynFormula1} and \eqref{e:InverseKasteleynFormula2}, determine the dimer-dimer correlation function ${\rm Cov}_{L,N,\underline w}(e,e')$.
At this point we can take the infinite-volume limit:
\begin{proposition}
  \label{prop:formuloneCov}
  Let  $e=\{(x_1,y_1),(x_1+1,y_1)\},e'=\{(x_2,y_2),(x_2+1,y_2)\}$ be two horizontal edges of $T_{L, N}$ with $(x_1,y_1)$ a black vertex and $(x_2,y_2)$ a white vertex.
  Assume without loss of generality $x_2=y_2=0$ and let $y:=y_1>0,x:=x_1$. Then, for $H_1\in [0,H_c),H_2=0$,
 \begin{multline}
   \label{formuloneCov}
   \mathrm{Cov}_{\underline{w}}(e,e'):=   \lim_{L\to\infty}\lim_{N\to\infty}\mathrm{Cov}_{L, N, \underline{w}}(e,e')=\frac4{\pi^2}(-1)^{y+1}w_1(y)w_1(0)\\
   \times\int_{\theta_c}^{\pi/2}\cos\Big(x \theta-\frac\pi2(x\!\!\!\!\mod 4)\Big)
   \frac{\langle e_1,V_{<0}^\theta\rangle\langle V_{>y}^\theta,e_1\rangle\prod_{r=0}^{y-1}(e^{H_1  }w_2(r))}{\langle
M_0^\theta\dots,M^\theta_y V_{>y}^\theta,V_{<0}^\theta\rangle
   }\dd \theta \\
   \times\left[
(-1)^{y+1} \int_0^{\theta_c}\cos\Big(x \theta-\frac\pi2(x\!\!\!\!\mod 4)\Big)\frac{\langle M_1^\phi\dots,M^\phi_{y-1}e_1,e_1 \rangle}{\prod_{r=1}^{y-1}(e^{H_1}  w_2(r))}\dd \phi \right.\\\left.+\int_{\theta_c}^{\pi/2}\cos\Big(x \phi-\frac\pi2(x\!\!\!\!\mod 4)\Big) \frac{ \langle e_1,V_{<0}^\phi\rangle\langle V_{>y}^\phi,e_1\rangle\prod_{r=0}^{y-1}(e^{-H_1  }w_2(r))}{\langle
M_0^\phi\dots,M^\phi_y V_{>y}^\phi,V_{<0}^\phi\rangle
   } \dd \phi
     \right]
 \end{multline}
 where $M^\theta$ is defined in \eqref{e:Matrixnew}, $
 \theta_c=\theta_c(H_1):=\cL^{-1}(H_1)$ and, for $v\in (\R^+)^2 \setminus \{0\}$,
 \begin{equ}
   \label{eq:vR}
   V_{>y}^\theta\eqdef \lim_{N\to\infty}\frac{M^\theta_{y+1}\dots M^\theta_{N}v}{\|M^\theta_{y+1}\dots M^\theta_{N}v\|_2}, \quad
   V_{<0}^\theta\eqdef \lim_{N\to\infty}\frac{(M^\theta_{-1})^*\dots (M^\theta_{-N})^*v}{\|(M^\theta_{-1})^*\dots (M^\theta_{-N})^*v\|_2}.
 \end{equ}
For $\theta \in (0,\pi/2]$ the (random) limits $V_{>y}^\theta,V_{<0}^\theta$ exist almost surely, have positive coordinates, are independent of $v$ and are continuous in $\theta$.
\end{proposition}
In particular, when $H_1 = H_2 = 0$ \eqref{formuloneCov} reduces to
\begin{multline}
   \label{formuloneCovbeta0}
   \mathrm{Cov}_{\underline{w}}(e,e'):=   \lim_{L\to\infty}\mathrm{Cov}_{L, N, \underline{w}}(e,e')=\frac4{\pi^2}
(-1)^{y+1}w_1(y)w_1(0)\\
   \times\left(\int_{0}^{\pi/2}\cos\Big(x \theta-\frac\pi2(x\!\!\!\!\mod 4)\Big)
   \frac{w_2(0)\dots w_2(y-1)\langle e_1,V_{<0}^\theta\rangle\langle V_{>y}^\theta,e_1\rangle}{\langle
M_0^\theta\dots,M^\theta_y V_{>y}^\theta,V_{<0}^\theta\rangle
   }\dd \theta \right)^2.
 \end{multline}

 \begin{proof}[of Proposition \ref{prop:formuloneCov}]
   Note that $M^\theta=\M^{2\sin(\theta)}$ with the notation \eqref{e:Matrixnewz}.
The existence of the limits \eqref{eq:vR} and their independence of $v\in (\R^+)^2$ follows directly from Corollary \ref{cor:map}.
As for continuity: from the same corollary, we know that $T_{M^\theta_{y+1}}\dots T_{M^\theta_{N}}(\H)$ shrinks exponentially fast in $N$  to a limit point $z_\theta$, uniformly in the disorder realization and in $\theta\ge \theta_0>0$. Take $N_0=N_0(\epsilon)$ large enough so that the diameter of $T_{M^\theta_{y+1}}\dots T_{M^\theta_{N}}(\H)$ is smaller than $\eps/2$ for $\theta\ge \theta_0,N\ge N_0$. On the other hand, for fixed $N$ the range $T_{M^\theta_{y+1}}\dots T_{M^\theta_{N}}(\H)$ is continuous in $\theta$. It follows that $|z_{\theta}-z_{\theta'}|\le \eps$ if $\theta,\theta'>\theta_0$ are close enough. This is equivalent to continuity of $\theta\mapsto V_{>y}^\theta$.

   As for \eqref{formuloneCov}, we discuss separately the cases $H_1=0$ and $H_1>0$, that present some technical differences. In both cases, the starting point is \eqref{e:CovKasteleynTheory}, where $e=(\rb_1,\rw_1),e'=(\rb_2,\rw_2)$.

   \underline{Proof of  \eqref{formuloneCov} for $H_1=0$.}
   First we show that
   \begin{equation}
     \label{f0}
     \lim_{L\to\infty}\lim_{N\to\infty}\prod_{j=1,2} \LLangle K^{-1}_\tau(\rb_j,\rw_j) \RRangle=     \lim_{L\to\infty}\lim_{N\to\infty}\LLangle \prod_{j=1,2} K^{-1}_\tau(\rb_j,\rw_j) \RRangle
   \end{equation}
   by computing separately the two limits.
   We start with $\LLangle K^{-1}_\tau(\rb_1,\rw_1) \RRangle$.
   Using Eqs. \eqref{e:InverseKasteleyn}, \eqref{e:InverseKasteleynFourier},
\eqref{e:CofactorIdentity} and \eqref{e:InverseKasteleynFormula3}
   we see that (letting $\theta:=\frac{\pi k}L$)
   \begin{equs}
     \label{f1}
     K^{-1}_\tau(\rb_1,\rw_1) &= \frac 1L \sum_{k\in F_{\tau_1}}e^{-\iota \theta}
     \frac{\sum_{a=1}^2\langle R^\theta_{-N+1}\dots R_{y-1}^\theta e_1,e_a\rangle\langle R^\theta_{y+1}\dots R_{N}^\theta e_a,e_1\rangle}{D(\theta,\tau)} \\
     \mbox{where } D(\theta,\tau) &= \sum_{a=1}^2 \langle R^\theta_{-N+1} \dots R_N^\theta e_a,e_a\rangle + 2(-1)^{\tau_2}\prod_{r=-N+1}^Nw_2(r).
     \label{Dthetatau}
   \end{equs}
   We know from Lemma \ref{lem:GeneralLyap} that ${\rm Lyap}(R^\theta)>0$ except if $\cos(\theta)=0$, that is, if  $k=L/2$. The term $k=L/2$ (that belongs to $F_{\tau_1}$  only if $\tau_1=0$) will be treated separately later.
   For $k \ne (L/2)$, the term proportional to $(-1)^{\tau_2}$ in $D(\theta,\tau)$ is almost surely $\exp(o(N))$ by the law of large numbers and since $\log w_2$ is centered, while the other grows almost surely as $\Theta(1) \, e^{2 N \, {\rm Lyap}(R^\theta)}$ since it is easy to check that for instance 
   \begin{equ}
   \langle R^\theta_{-N+1} \dots R_N^\theta e_a, e_a\rangle \gtrsim \theta^4 \, \|(R^\theta_{-1})^*\dots(R^\theta_{-N+1})^*\| \|R^\theta_1\dots R^\theta_N\| \, .
   \end{equ}
   Therefore, the quotient in \eqref{f1} for $k \ne L/2$ is almost surely equivalent, as $N \to +\infty$, to
   \begin{eqnarray}\label{f1-}
      \frac{\sum_{a=1}^2\langle  e_1,v^{a,N,\theta}_{<y}\rangle\langle v^{a,N,\theta}_{>y},e_1\rangle
      \|(R^\theta_{y-1})^*\dots(R^\theta_{-N+1})^* e_a \|
      \|R^\theta_{y+1} \dots R^\theta_N e_a \|}{\sum_{a=1}^2 \langle R^\theta_{y} v^{a,N,\theta}_{>y},v^{a,N,\theta}_{<y}\rangle \|(R^\theta_{y-1})^*\dots(R^\theta_{-N+1})^* e_a \| \|R^\theta_{y+1} \dots R^\theta_N e_a \|},
   \end{eqnarray}
   where here and in the rest of the proof we write $\|\cdot\|$ for the Euclidean norm and
   \begin{eqnarray}
     \label{eq:f2}
     v^{a,N,\theta}_{<y}:=\frac{ (R^\theta_{y-1})^*\dots(R^\theta_{-N+1})^* e_a }{\|(R^\theta_{y-1})^*\dots(R^\theta_{-N+1})^* e_a \|},\qquad
     v^{a,N,\theta}_{>y}:=\frac{R^\theta_{y+1}\dots R^\theta_N e_a}{\|R^\theta_{y+1}\dots R^\theta_N e_a\|}  .
   \end{eqnarray}
   All terms are non-negative, thus the ratio \eqref{f1-} belongs to the interval
   \[
     \left[
\min_{a=1,2} \frac{\langle  e_1,v^{a,N,\theta}_{<y}\rangle\langle v^{a,N,\theta}_{>y},e_1\rangle}{\langle R^\theta_{y} v^{a,N,\theta}_{>y},v^{a,N,\theta}_{<y}\rangle},\max_{a=1,2}\frac{\langle  e_1,v^{a,N,\theta}_{<y}\rangle\langle v^{a,N,\theta}_{>y},e_1\rangle}{\langle R^\theta_{y} v^{a,N,\theta}_{>y},v^{a,N,\theta}_{<y}\rangle}
       \right].
     \]
    Since $v^{a,N,\theta}_{\gtrless y}\to V^{\pi/2-\theta}_{\gtrless y}$ as $N\to\infty$ (recall
    that $R^\theta=M^{\frac\pi2-\theta}$), the matrix $R^\theta$ have positive top left entry, and $\theta \mapsto R^\theta$ is even (and therefore $\theta\mapsto V^\theta_{\gtrless y}$ as well), the sum in the r.h.s. of \eqref{f1} (excluding the possible value $k=L/2$) converges in the  limit  $N\to\infty$ to
     \begin{eqnarray}
       \label{f3}
       \frac1L \sum_{k\in F_{\tau_1}\setminus\{L/2\}}\cos(\theta) \, \frac{\langle  e_1,V^{\pi/2-\theta}_{<y}\rangle\langle V^{\pi/2-\theta}_{>y},e_1\rangle}{\langle R^\theta_{y} V^{\pi/2-\theta}_{>y},V^{\pi/2-\theta}_{<y}\rangle}  \end{eqnarray}
At this point we can let $L\to\infty$: since the summand in \eqref{f3} is upper bounded by $1/(2 w_1(y))$ and the weights are bounded away from zero, the sum tends to the integral
\begin{eqnarray}
  \label{f4}
\frac2\pi\int_0^{\pi/2}\sin(\theta)\frac{\langle  e_1,V^{\theta}_{<y}\rangle\langle V^{\theta}_{>y},e_1\rangle}{\langle M^\theta_{y} V^{\theta}_{>y},V^{\theta}_{<y}\rangle}\dd \theta.
\end{eqnarray}
     Now, we consider the special case $k=L/2$. In this case, $\theta=\pi/2$ and the matrix $R^\theta$ is  antidiagonal. Therefore, a product of copies of matrices $R^\theta$ is diagonal or antidiagonal according to the parity of the number of matrices and it is immediately seen that the numerator in \eqref{f1} vanishes identically.
     Wrapping up: we have that $\lim_L\lim_N \det (K^{-1}_\tau(\rb_1,\rw_1))$ equals the r.h.s. of \eqref{f4} for every $\tau$. Since the ratios $(\det K_\tau)/Z_{L,N}$ are bounded by $1$ as noted in \eqref{eq:factor2} and $\LLangle 1\RRangle=1$, we see that
     $\lim_L\lim_N  \LLangle K^{-1}_\tau(\rb_1,\rw_1) \RRangle$ is given by \eqref{f4}. The limits of $\LLangle K^{-1}_\tau(\rb_2,\rw_2) \RRangle$ and
     $\LLangle \prod_{j=1,2} K^{-1}_\tau(\rb_j,\rw_j) \RRangle$ are computed similarly (the former equals \eqref{f4} with $y$ replaced by $0$, the latter by the product of the two integrals), and \eqref{f0} follows.

     At this point, to prove the proposition for $H_1=0$ it remains to prove that the limit of  $\LLangle  K^{-1}_\tau(\rb_1,\rw_2)K^{-1}_\tau(\rb_2,\rw_1) \RRangle$ is $(-1)^y\frac
     {4}{\pi^2}$ times the squared integral in \eqref{formuloneCovbeta0}.
Using Eqs. \eqref{e:InverseKasteleyn}, \eqref{e:InverseKasteleynFourier},
\eqref{e:CofactorIdentity}, \eqref{e:InverseKasteleynFormula1} and  \eqref{e:InverseKasteleynFormula2} we see that (with $D(\theta,\tau)$ as in \eqref{Dthetatau})
\begin{equs}
  K^{-1}_\tau(\rb_1,\rw_2)K^{-1}_\tau&(\rb_2,\rw_1)= (-1)^y
\frac1{L^2}\sum_{k\in F_{\tau_1}}e^{\iota \frac{\pi k}L(x_1-x_2)}\frac{N(\frac{\pi k}L,\tau)}{D(\frac{\pi k}L,\tau)}\sum_{p\in F_{\tau_1}}e^{\iota \frac{\pi p}L(x_2-x_1)}\frac{N(\frac{\pi p}L,\tau)}{D(\frac{\pi p}L,\tau)}
\\
 N(\theta,\tau) &= w_2(0)\dots w_2(y-1)\sum_{a=1}^2 \langle R^\theta_{-N+1}\dots R^\theta_{-1}e_1,e_a\rangle \langle R^\theta_{y+1}\dots R^\theta_Ne_a,e_1\rangle \\
&\quad - (-1)^{\tau_2} (-1)^y \langle R^\theta_1\dots R^\theta_{y-1}e_1,e_1
\rangle\prod_{r=-N+1}^{-1}w_2(r)\prod_{r=y}^N w_2(r).  \label{absval}
\end{equs}
As in the computation of $\LLangle K^{-1}_\tau(\rb_j,\rw_j) \RRangle$, whenever $\theta\ne \pi/2$ we can neglect the term proportional to $(-1)^{\tau_2}$ in $D(\theta,\tau)$, but also in $N(\theta, \tau)$ by the same argument. Then, proceeding as above, one easily concludes that the limit $\lim_{L\to\infty} \lim_{N\to\infty}$ of \eqref{absval} with $k,p$ restricted to $F_{\tau_1}\setminus\{L/2\}$ is $(-1)^y\frac
     {4}{\pi^2}$ times the squared integral in \eqref{formuloneCovbeta0}.
     It remains to show that the contribution from the terms where one among $k,p$ equals $L/2$ in \eqref{absval} is negligible. Let us show this for instance when  $k=p=L/2$ (the case $p=L/2\ne k$ is similar and actually simpler).
     As noted above, even/odd products of $R^{\pi/2}$ are diagonal/antidiagonal. Then, it is easily seen that $N(\pi/2,\tau)=0$ whenever $y$ is even.
     If instead $y$ is odd, then an elementary computation gives
     \begin{equs}
       D(\pi/2,\tau) = (E+(-1)^{\tau_2}O)^2 \, , \quad
       &N(\pi/2,\tau) = E (E+(-1)^{\tau_2}O) Q \, , \quad \mbox{where}\\
       E:=\prod_{-N<r\le N:r\in 2\N}w_2(r), \quad O&:=\prod_{-N<r\le N:r\in 2\N+1}w_2(r) ,\quad Q:=\frac{w_2(1)w_2(3)\dots w_2(y-2)}{w_2(0)w_2(2)\dots w_2(y-1)}
     \end{equs}
     and therefore in the r.h.s. of \eqref{absval} we have the term
     \begin{eqnarray}
       \frac1{L^2} \, Q^2 \frac{E^2}{(E+(-1)^{\tau_2}O)^2}.
     \end{eqnarray}
     Since we are computing $\LLangle K^{-1}_\tau(\rb_1,\rw_2)K^{-1}_\tau(\rb_2,\rw_1)\RRangle$,  we get the linear combination
     \begin{eqnarray}
       \label{lincomb}
       \frac1{2L^2}\sum_\tau c(\tau) \frac{\det(K_\tau)}{Z_{L,N}}Q^2\frac{E^2}{(E+(-1)^{\tau_2}O)^2}.
     \end{eqnarray}
     In view of \eqref{eq:factor2},
     the only possibly problematic term is $\tau_2=1$,$\tau_1=0$ (the latter since otherwise the value $k=L/2$ does not belong to $F_{\tau_1}$), while the others tend to zero as $L\to\infty$.
     On the other hand, using the notation \eqref{detKtauk} and noticing in particular that $\det(K_{L/2}^{\tau_2 = 1}) = (E - O)^2$, that $|\det(K_k^{\tau_2 = 1})| \leq \det(K_k^{\tau_2 = 0})$ and that $\det(K_{L/2}^{\tau_2 = 1}) = (E + O)^2$, we deduce
     \begin{equs}
       \frac{\det(K_{0,1})}{Z_{L,N}} \frac{E^2}{(E - O)^2} &= \frac{E^2}{Z_{L,N}} \prod_{k \in F_1 \setminus \{L/2\}}\det (K^{\tau_2=1}_k) \\
       &\le \frac{(E+O)^2}{Z_{L,N}} \prod_{k\in F_1\setminus \{L/2\}} \det (K^{\tau_2=0}_k) \\
       &= \frac{\det (K^{\tau_2=0}_k)}{Z_{L,N}} =\frac{\det(K_{0,0})}{Z_{L,N}} \le 1.
     \end{equs}
     In conclusion, all the terms in \eqref{lincomb} tend to zero as $L\to\infty$.

   \underline{Proof of  \eqref{formuloneCov} for $H_1\in(0,H_c)$.}
   This is similar to the $H_1=0$ case, except for the issue of the special values of $k$ that might cause the denominators to vanish. Therefore, we point out only the relevant differences.
   Let us prove that the limit of  $\LLangle  K^{-1}_\tau(\rb_1,\rw_2)K^{-1}_\tau(\rb_2,\rw_1) \RRangle$ is $(-1)^y\frac
   {4}{\pi^2}$ times the product of two integrals in \eqref{formuloneCov}.
   An analogous (and, actually, simpler) proof gives that the limit of
   $\LLangle  K^{-1}_\tau(\rb_1,\rw_1)\RRangle \LLangle K^{-1}_\tau(\rb_2,\rw_2) \RRangle-\LLangle  K^{-1}_\tau(\rb_1,\rw_1)K^{-1}_\tau(\rb_2,\rw_2) \RRangle$
   is zero. Using Eqs. \eqref{e:InverseKasteleyn}, \eqref{e:InverseKasteleynFourier},
   \eqref{e:CofactorIdentity}, \eqref{e:InverseKasteleynFormula1} and  \eqref{e:InverseKasteleynFormula2} we see that
   \begin{equs}
  K^{-1}_\tau(\rb_1,\rw_2)&K^{-1}_\tau(\rb_2,\rw_1)=
\frac{(-1)^y}{L^2}\sum_{k\in F_{\tau_1}}e^{\iota \frac{\pi k}L(x_1-x_2)}\frac{N_+(\frac{\pi k}L,\tau)}{D(\frac{\pi k}L,\tau)}\sum_{p\in F_{\tau_1}}e^{\iota \frac{\pi p}L(x_2-x_1)}\frac{N_-(\frac{\pi p}L,\tau)}{D(\frac{\pi p}L,\tau)}
 \, , \\
N_\pm(\theta,\tau) &= e^{\pm H_1 y}w_2(0)\dots w_2(y-1)\sum_{a=1}^2 \langle R^\theta_{-N+1}\dots R^\theta_{-1}e_1,e_a\rangle \langle R^\theta_{y+1}\dots R^\theta_Ne_a,e_1\rangle\\
&- (-1)^{\tau_2} (-1)^y \langle R^\theta_1\dots R^\theta_{y-1}e_1,e_1
  \rangle\prod_{r=-N+1}^{-1}(e^{\mp H_1}w_2(r))\prod_{r=y}^N (e^{\mp H_1}w_2(r))
  \, , \\
  D(\theta,\tau) &= \sum_{a=1}^2 \langle R^\theta_{-N+1}\dots R_N^\theta e_a,e_a\rangle+2(-1)^{\tau_2}\cosh(2NH_1)\prod_{r=-N+1}^Nw_2(r).
\end{equs}
The two terms in $D(\theta,\tau)$ grow like $e^{2N {\rm Lyap}(R^\theta)}$ and $e^{2N H_1}$, respectively and the two exponential rates coincide for $\theta=\bar\theta:=\frac\pi2-\theta_c(H_1)$.
Consider first the case where none of the $F_{\tau_1}$ contains $k$ such that $k\pi/L=\bar\theta$. In this case, up to negligible (exponentially small in $N$) error terms,
\begin{equs}
  \sum_{k\in F_{\tau_1}} &e^{\iota \frac{\pi k}L(x_1-x_2)}\frac{N_+(\frac{\pi k}L,\tau)}{D(\frac{\pi k}L,\tau)} \\
  &\approx \prod_{r=0}^{y-1}(e^{H_1}w_2(r)) \times \sum_{\substack{k=\frac{L\theta}\pi\in F_{\tau_1}:\\|\theta|<\bar \theta}}
  e^{\iota \theta(x_1-x_2)}\frac{\sum_{a=1}^2 \langle R^\theta_{-N+1}\dots R^\theta_{-1}e_1,e_a\rangle \langle R^\theta_{y+1}\dots R^\theta_Ne_a,e_1\rangle}{\sum_{a=1}^2 \langle R^\theta_{-N+1}\dots R_N^\theta e_a,e_a\rangle} \\
  \sum_{p\in F_{\tau_1}} &e^{\iota \frac{\pi p}L(x_2-x_1)}\frac{N_-(\frac{\pi p}L,\tau)}{D(\frac{\pi p}L,\tau)} \\
  &\approx \sum_{\substack{p=\frac{L\phi}\pi\in F_{\tau_1}:\\|\phi|<\bar \theta}}
   e^{\iota \phi(x_2-x_1)}\frac{\prod_{r=0}^{y-1}(e^{-H_1 }w_2(r))\sum_{a=1}^2 \langle R^\phi_{-N+1}\dots R^\phi_{-1}e_1,e_a\rangle \langle R^\phi_{y+1}\dots R^\phi_Ne_a,e_1\rangle}{\sum_{a=1}^2 \langle R^\phi_{-N+1}\dots R_N^\phi e_a,e_a\rangle}
  \\
  &\qquad + (-1)^{y+1} \sum_{\substack{p=\frac{L\phi}\pi\in F_{\tau_1}:\\|\phi|>\bar \theta}} e^{\iota \phi(x_2-x_1)}\frac{\langle R_1^\phi\dots,R^\phi_{y-1}e_1,e_1 \rangle}{\prod_{r=1}^{y-1}(e^{H_1}  w_2(r))}.
\end{equs}
At this point, one proceeds like in Eqs. \eqref{f1-}-\eqref{f4}, and \eqref{formuloneCov} follows (recall that $\theta\mapsto R^\theta$ is even and that $R^{\pi/2-\theta}=M^\theta,\pi/2-\bar\theta=\theta_c(H_1)$).

It remains to consider the case where $\tau_2=1$ and $\pm L\bar\theta/\pi\in F_{\tau_1}$, so that  $D(\pm \bar\theta,\tau)$ may potentially grow exponentially slower than either of the two individual terms. Specifically, we consider the case where $k=p=L\bar\theta/\pi$, the cases where only one of the two equals $\pm \bar\theta/\pi$ being similar and easier.
We want to show that (with $\tau=(\tau_1,1)$)
\begin{eqnarray}
  \label{fsometh}
  \frac1{L^2}\frac{\det(K_{\tau})}{Z_{L,N}} \frac{N_+(\bar\theta,\tau)N_-(\bar\theta,\tau)}{D(\bar \theta,\tau)^2}
\end{eqnarray}
tends to zero in the limit $\lim_{L\to\infty}\lim_{N\to\infty}$. In fact, this equals
\begin{multline}
  \frac1{L^2}\frac{\prod_{k\in F_{\tau_1}:k
      \pi/L\ne\pm\bar\theta}\det(K^{\tau_2=1}_{k})}{Z_{L,N}}N_+(\bar\theta,\tau)N_-(\bar\theta,\tau)\\
  \le \frac1{L^2}\frac{\det(K_{(\tau_1,0)})}{Z_{L,N}}\frac{N_+(\bar\theta,(\tau_1,0))N_-(\bar\theta,(\tau_1,0))}{D(\bar \theta,(\tau_1,0))^2}
\end{multline}
The last ratio is easily seen to be uniformly bounded in $N$ and $L$.
Together with \eqref{eq:factor2}, this implies that \eqref{fsometh} tends indeed to zero.
 \end{proof}

\section{Proof of the main results} \label{s:proofs}

First, notice that $\cL(\theta, H_2) = L(2 \sin(\theta + \iota H_2)) = L(\iota e^{H_2} e^{- \iota \theta} - \iota e^{-H_2} e^{\iota \theta})$, and that $\theta \in [0, \pi/2] \mapsto 2 \sin(\theta + \iota H_2)$ describes a quarter of ellipse starting at $2\iota \sinh(H_2)$ and ending at $2\cosh(H_2)$. Thus Lemmas \ref{lem:GeneralLyap} and \ref{lem:miracolo} show that the function $\theta \in [0,\pi/2] \to \cL(\theta, H_2)$ is non-negative, strictly increasing, continuous, and it is real analytic on $(0,\pi/2]$.

\subsection{Proof of Theorem \ref{th:F}}

The claim immediately follows from \eqref{formulaF0} together with $\cL(\theta) = L(2\sin(\theta))$ (using the notation \eqref{e:Matrixnewz}) and the asymptotic \eqref{eq:Lreax} for $L$ near zero. In the case where the law of $\log w_2$ admits a density, the sharper \eqref{eq:Fessential2} follows from \eqref{eq:sharper} together with \eqref{suegiu}.

\subsection{Proof of Theorem \ref{th:Fbetane0}}

Notice that
\begin{equation}
  \Delta(H_2) = L(2\iota\sinh(H_2)) = {\rm Lyap}(V^{2\sinh(H_2)}) \, ,
\end{equation}
where $L(\M^{\iota x}) = \mathrm{Lyap}(V^x)$ by \eqref{Vx}. Then, the statement follow immediately from Lemma \ref{lem:GeneralLyap}.

\subsection{Proof of Theorem \ref{th:32}}

Using the general properties of $\cL(\cdot, H_2)$, we can write for $H_1 \in (\Delta(H_2), H_c(H_2))$
\begin{eqnarray}
  F(H_1,H_2) &= \frac{H_1}2+
  \frac1\pi\int_{\cL^{-1}(H_1,H_2)}^{\pi/2} (\cL(\theta,H_2)-H_1) \, \dd \theta \, , \\
  &= \frac{H_1}2 +
  \frac1\pi \int_{H_1}^{H_c(H_2)} (\pi/2 - \cL^{-1}(y, H_2) ) \, \dd y \, , \label{eq:strpart}
\end{eqnarray}
By Lemma \ref{lem:miracolo} we have the asymptotic, as $\theta \nearrow \pi/2$,
\begin{equ}
\cL(\pi/2, H_2) - \cL(\theta, H_2) \sim c \, (\pi/2 - \theta)^2 \, ,
\end{equ}
for some constant $c = c(H_2) > 0$. Hence, we deduce the asymptotic, as $y \nearrow H_1$,
\begin{eqnarray}
\pi/2 - \cL^{-1}(y, H_2) \sim c^{-1/2}  \, (H_c(H_2) - y)^{1/2} \, ,
\end{eqnarray}
which, together with \eqref{eq:strpart}, gives, as $H_1 \nearrow H_c(H_2)$
\begin{equ}
F(H_1, H_2) - \frac{H_1}{2} \sim \frac{2 c^{-1/2}}{3 \pi} \, (H_c(H_2) - H_1)^{3/2} \, ,
\end{equ}
and concludes the proof of \eqref{32dis}.

\subsection{Proof of Theorem \ref{th:corrdecay}}
\label{sec:corrdecay}

At exponential scale, the matrix elements of  $M^\theta_0\dots M_y^\theta$ grow like $e^{y \cL(\theta)}$. If we could perform this replacement in \eqref{formuloneCov} and \eqref{formuloneCovbeta0} then the proof of the Theorem would be rather direct. The problem is that we would need a quantitative, finite-size version of this, and the above asymptotic is unlikely to hold uniform in $\theta$ for $\theta\approx 0$, which is especially relevant for \eqref{formuloneCovbeta0}. Therefore, we have to follow a more probabilistic approach instead.

From here on,  $C$ denotes a positive constant that  depends only on the law of weights, and whose value may change from line to line.

For simplicity we assume that  $y$ is odd and we write $y=2n+1$.
  It is convenient to define
  \begin{eqnarray}
    \label{eq:NiMi}
     T^\theta_i=M^\theta_{2i}M^\theta_{2i+1},
  \end{eqnarray}
  so that
 $M_0^\theta\dots
  M_y^\theta=T_0^\theta\dots T_n^\theta$ (if $y$ were even instead, we would have a product of $y/2$ matrices $T^\theta_i$ times an additional matrix $M^\theta_{y}$, which would only make formulas more cumbersome).
By \eqref{suegiu} and $M^\theta=\M^{2\sin(\theta)}$ we have the bounds (intended to hold entry-wise, for some universal constant $C\in(0,\infty)$ independent of the disorder realization and of $\theta\in[0,\pi/2]$)
  \begin{eqnarray}
    \label{eq:Ni}
    \begin{pmatrix}
      w_2(2i)^2& C^{-1}\theta\\
      C^{-1}\theta & w_2(2i+1)^2 \end{pmatrix}\le T^\theta_i\le (1+C\theta^2) \begin{pmatrix}
      w_2(2i)^2& C\theta\\
      C\theta & w_2(2i+1)^2
    \end{pmatrix}.
  \end{eqnarray}
  We define the random walks $S_+,S_-,S$ as $S_+(-1)=S_-(-1)=S(-1)=0$ and
  \begin{eqnarray}
    S_+(i)=\sum_{j=0}^{i} \log w_2(2j),\quad
    S_-(i)=\sum_{j=0}^{i} \log w_2(2j+1),\quad
  S(i)=S_+(i)-S_-(i).
  \end{eqnarray}
  The three random walks have centered and bounded increments, with finite and non-zero variance. Also, $S_+,S_-$ are independent. For convenience, we set
  \begin{eqnarray}
    \label{eq:Ds}
    \Delta S_\pm(i,j)\eqdef S_\pm(j)-S_\pm(i), \;\Delta S(i,j)\eqdef S(j)-S(i), \quad j\ge i.
  \end{eqnarray}
With this notation,
\begin{eqnarray}
  \label{eq:prodw2}
w_2(0)\dots w_2(y-1)= e^{S_+(n)+S_-(n)}/w_2(2n+1)\asymp e^{S_+(n)+S_-(n)}.
\end{eqnarray}
Because of Lemma \ref{lem:duedischi}, the components of the vectors $V^\theta_{<0},V^\theta_{>y}$ satisfy the bounds
\begin{eqnarray}
  \label{eq:Vnosmall}
  C^{-1} \theta\le (V^\theta_{<0})_{a}\le 1, \quad C^{-1}\theta\le (V^\theta_{>y})_{a}\le 1, \quad a=1,2.
\end{eqnarray} Also, from \eqref{eq:Ni},
\begin{eqnarray}
  \label{eq:Ni2}
  \theta^2\lesssim \frac{(T_0^\theta\dots T_n^\theta)_{ij}}{(T_0^\theta\dots T_n^\theta)_{i'j'}}\lesssim \theta^{-2}, \quad  i,j,i',j'\in \{1,2\},
\end{eqnarray}
uniformly in $n$ and in the disorder realization.

Since both the
vectors $V^\theta_{<0},V^\theta_{>y}$ and the matrices $T_i^\theta$
have non-negative entries,
\begin{equs}
  \label{tengo11}
  \langle T_0^\theta\dots T_n^\theta V_{>y}^\theta,&V_{<0}^\theta\rangle \ge \langle e_1,V_{<0}^\theta\rangle \langle V_{>y}^\theta,e_1\rangle \langle T_0^\theta\dots T_n^\theta e_1,e_1\rangle\\
  \label{tengoTr}
 \theta^4 \lesssim &\frac{\langle T_0^\theta\dots T_n^\theta V_{>y}^\theta,V_{<0}^\theta\rangle }{{\rm Tr}( T_0^\theta\dots T_n^\theta )\langle e_1,V_{<0}^\theta\rangle \langle V_{>y}^\theta,e_1\rangle}\lesssim \theta^{-4}
\end{equs}
where in \eqref{tengoTr} we used also \eqref{eq:Vnosmall} and \eqref{eq:Ni2}.

\subsubsection{Case $H_1=0$}
For the upper bound in \eqref{eq:corrstretched}, we use \eqref{eq:prodw2} and \eqref{tengo11} to upper bound the integral in
\eqref{formuloneCovbeta0} as
\begin{eqnarray}
  \label{ub1}
 C e^{S_+(n)+S_-(n)}\int_0^{\pi/2}\frac{1}{\langle T_1^\theta\dots T_n^\theta e_1,e_1\rangle}\dd \theta.
\end{eqnarray}
Observe that
\begin{equ}
\langle T_1^\theta\dots T_n^\theta e_1,e_1\rangle = \sum_{\gamma_1, \dots, \gamma_{n+1} \in \{1,2\}, \, \gamma_1 = \gamma_{n+1} = 1} \prod_{i=0}^{n} (T^\theta_i)_{\gamma_i, \gamma_{i+1}} \, .
\label{paths}
\end{equ}
Let $1 \leq i < j \leq n$ to be chosen later. Recalling \eqref{eq:Ni} and keeping only the two terms corresponding to $\gamma$ being constant or flipping at $i$ and $j$, we deduce
\begin{eqnarray}
\label{possoscegliere}
\langle T_1^\theta\dots T_n^\theta e_1,e_1\rangle\ge e^{2 S_+(n)} + C \theta^2 e^{2 S_+(n) - 2 \Delta S(i,j)}
\end{eqnarray}
Then, \eqref{ub1} is upper bounded by
\begin{eqnarray}
\label{eq:ub2}
C e^{-S(n)} \int_0^{\pi/2} \frac{1}{1+C \theta^2 e^{-2 \Delta S(i,j)}} \dd \theta \, ,
\end{eqnarray}
Taking $j = n$ and using that $\int_0^{+\infty} \frac{\dd u}{1 + \kappa u^2} = \frac{\pi}{2\sqrt{\kappa}}$, we can upper \eqref{eq:ub2} by $C e^{-S(i)}$. Taking the optimal $i$ leads to
\begin{equ}
|\mathrm{Cov}_{\underline \omega}(e,e')| \leq C e^{- \max_{1 \leq i < n} S(i)} \, ,
\end{equ}
which concludes, since by the invariance principle $\frac{1}{\sqrt{n}} \max_{1 \leq i < n} S(i)$ converges in distribution to a strictly positive random variable.

Next, we prove the lower bound  \eqref{eq:stretched4}. Call $I$ the integral in \eqref{formuloneCovbeta0}.
Let $\theta_+>0$ denote the minimum between $1/2$ and the leftmost zero of $\theta\mapsto \cos(\theta x-\frac\pi2(x\!\!\mod 4))$ on $[0,\pi/2]$ (the minumum is there simply to guarantee that $\log(1/\theta_+)>0$). Under the assumptions of the theorem, $\theta_+\gtrsim {\delta}/|x|\gtrsim e^{-\eps_n n^{1/2}}$. Split $I=I'+I''$ with $I'$ the integral restricted to $[0,\theta_+]$.
We first show that $I''$ is negligible, that is $|I''|\le e^{-U_n}$ with $U_n\gg \sqrt n$. First of all, we have
\begin{equ}
|I''| \leq C \max_{\theta \in [\theta_+, \pi/2]} \frac{e^{  S_+(n) + S_-(n)}}{\langle T_0^\theta \dots T_n^\theta e_1, e_1 \rangle}.
\end{equ}
For any $i \leq j$, expanding the product $T_i^\theta\dots T_j^\theta$ as in \eqref{paths} and keeping as above only the paths $\gamma$ which are either constant or flip exactly at $i$ and $j$, we have
\begin{equs}
e^{- \Delta S_+(i,j) - \Delta S_-(i,j)} \, \langle T_i^\theta \dots T_j^\theta e_1, e_1 \rangle &\geq C^{-1} e^{- \Delta S_+(i,j) - \Delta S_-(i,j)} (e^{2 \Delta S_+(i,j)} + \theta_+^2 e^{2 \Delta S_-(i,j)}) \\
&\geq C^{-1} (e^{\Delta S(i,j)} + \theta_+^2 e^{- \Delta S(i,j)}) \geq C^{-1} \theta_+^2 e^{|\Delta S(i,j)|} \, .
\end{equs}
As a consequence we have
\begin{equ}
e^{-  S_+(n) -  S_-(n)} \, \langle T_1^\theta \dots T_n^\theta e_1, e_1 \rangle \geq C^{-1} \theta_+^2 \prod_{i=0}^{\cal T-1} C^{-1} \theta_+^2 e^{|\Delta S(\tau_i, \tau_{i+1})|} \, ,
\end{equ}
for any integers $\cal T$ and $0=
\tau_0< \dots < \tau_{\cal T} < n$. Let us take $\tau_i$ as the (random) minimal integer such that
\begin{eqnarray}
  C^{-1} \theta_+^{2} e^{|\Delta S(\tau_{i-1}, \tau_i)|} \ge \theta_+^{-2} \, ,
\end{eqnarray}
and $\cal T$ as the (random) maximum index such that $\tau_{\cal T} < n$.
Then,
\begin{eqnarray}
e^{-  S_+(n) -  S_-(n)} \, \langle T_1^\theta \dots T_n^\theta e_1, e_1 \rangle \ge C^{-1} \theta_+^2 e^{\cal{T} \log(1/\theta_+^2)}.
\end{eqnarray}
To conclude note that, as $n \to \infty$, $\theta_+\gtrsim e^{-\epsilon_n\sqrt n}$ and, with high probability, $\cal T \log(1/\theta_+^2) \gtrsim \frac{n}{\log(1/\theta_+)} \gg \sqrt{n}$. The latter is a simple consequence of the fact that $\tau$ is a renewal process with $\E[\tau_1] \lesssim (\log(1/\theta_+))^2$ (by the invariance principle) and  of the lower bound on $\theta_+$.
\ms

Finally, we give a lower bound on the integral $I'$.
Using \eqref{eq:Vnosmall} and \eqref{eq:Ni2}, we have
\begin{equs}
  \label{ang}
  w_2(0)\dots w_2(y-1) &\langle e_1,V_{<0}^\theta\rangle\langle V_{>y}^\theta,e_1\rangle \gtrsim \theta^2 e^{S_+(n)+S_-(n)}\\
\langle
T_0^\theta \dots T^\theta_n V_{>y}^\theta,V_{<0}^\theta\rangle
&\lesssim \sum_{a,b=1}^2 \langle T_0^\theta \dots T^\theta_{n}e_a,e_b\rangle\lesssim \theta^{-2} \langle T_0^\theta\dots T^\theta_{n}e_1,e_1\rangle.
\end{equs}
From \eqref{eq:Ni},
\begin{eqnarray}
  \label{eq:Nii}
  \langle T_0^\theta\dots T^\theta_{n} e_1,e_1\rangle\le e^{n C \theta^2} e^{S_+(n)+S_-(n)}
  \sum_{0\le i\le n/2}\sum_{0 = n_0 \le n_1<\dots<n_{2i}\le n}\\
\times
(C\theta)^{2i}\prod_{j=1}^{i}\Big(e^{\Delta S(n_{2j-2},n_{2j-1})-\Delta S(n_{2j-1},n_{2j})}\Big)e^{\Delta S(n_{2i},n)} \, .
\end{eqnarray}
Let $Z(n):= \log n + \max_{0\le i<j\le n} |\Delta S(i,j)|>0$ and note
that this random variable is of order $n^{1/2}$.  Since
$|\cos(\theta x-\frac\pi2(x\!\!\mod 4))|$ does not change sign on
$[0,\theta_+]$ (assume w.l.o.g. that the cosinus is positive there),
we restrict the integral to
$\cal I=[0,e^{-5 \max(Z(n),\sqrt{\eps_n} n^{1/2})}]$, where the
cosinus is larger than, say, $\delta/2$. Note that with high
probability $e^{-5 Z(n)}<e^{-5 \sqrt{\eps_n} n^{1/2}}\ll\theta_+$ since
the latter is at least of order $e^{-\eps_n n^{1/2}}$. For values of
$\theta$ in this interval, we can bound \eqref{eq:Nii} with a
converging geometric series:
\begin{eqnarray}
  \label{eq:Niii}
  \langle T_0^\theta\dots,T^\theta_{n}e_1,e_1\rangle \lesssim e^{S_+(n)+S_-(n)}
  \sum_{i\ge 0}\left(n^2 e^{3 Z(n)} C^2 e^{-10 Z(n)}\right)^i\lesssim e^{S_+(n)+S_-(n)},
\end{eqnarray}
where the constants implicit in the bounds are independent of $n$ and of the disorder realization. Therefore, collecting \eqref{ang} and \eqref{eq:Niii}, we have that on the interval $\cal I$, the integrand of \eqref{formuloneCovbeta0} is lower bounded by a universal constant times $\delta \theta^4$.
This implies that
\begin{eqnarray}
  \label{eq:ang2}
  I'\gtrsim \delta e^{- \sqrt{n} W(n)}, \quad W(n)\eqdef 20 \max\Big(\frac {Z(n)}{\sqrt n},\sqrt{\eps_n}\Big).
\end{eqnarray}
Since $W(n)$ tends in distribution to a strictly positive and finite random variable, we have with high probability $I'\gg |I''|$, so that $I=I'+I''\gtrsim I'\gtrsim \delta e^{- \sqrt{n} W(n)}$ and \eqref{eq:stretched4} follows.

\subsection{Case $H_1>0$}

Write
\begin{eqnarray}
   \frac{\pi^2}{4w_1(0)w_1(y)} {\rm Cov}_{\underline w}(e,e')=I_1(I_2+(-1)^{y+1}I_3)
\end{eqnarray}
where $I_i,i=1,2,3$ denote the three integrals in \eqref{formuloneCov}.
First of all, we prove that $I_1 I_3$ decays exponentially fast with $n$, and is therefore negligible for \eqref{eq:corrstretched2} and \eqref{eq:stretched3}.
Indeed, using \eqref{tengoTr},
\begin{eqnarray}
  \label{eq:I1I3}
  |I_1 I_3|\le (\theta_c)^{-4}e^{2 S_+(n)+2S_-(n)}\left(\int_{\theta_c}^{\pi/2}\frac{\dd \theta}{{\rm Tr}(T_0^\theta\dots T_n^\theta)}
  \right)^2.
\end{eqnarray}
The random function
\begin{eqnarray}
  \label{eq:gn}
\theta\mapsto g_n(\theta):=\frac1n\log {\rm Tr}(T_0^\theta\dots T_n^\theta)
\end{eqnarray}
 converges $\P$-a.s. to $\cL(\theta)=L(2\sin(\theta))>0$, that is bounded away from zero on $[\theta_c,\pi/2]$, and the convergence is uniform over compact subsets of $(0,\pi/2]$, as follows from Lemma \ref{lem:alternative} and Remark \ref{rem:also}. Since $S_+,S_-$ are centered random walks with finite variance, the claimed exponential decay follows.

To prove  \eqref{eq:corrstretched2}, we use \eqref{tengoTr} to write with the notation \eqref{eq:gn}
\begin{eqnarray}
|I_1 I_2|\le (\theta_c)^{-4} \Big(\int_{\theta_c}^{\pi/2}e^{-n [g_n(\theta)-g_n(\theta_c)]}{\dd \theta}\Big) \left(\int_0^{\theta_c}e^{n [g_n(\phi)-g_n(\theta_c)]}\dd \phi
  \right).\end{eqnarray}
Since $\phi\mapsto g_n(\phi)$ is increasing, $g_n\to \cL$ uniformly together with all its derivatives and  $\partial_\theta\cL(\theta)>0$ over compact sets of $(0,\pi/2]$,
we can first restrict the second integral to $[\theta_c/2,\theta_c]$ (up to an additive error term decaying exponentially with $n$) and then bound
\begin{eqnarray}
  \label{eq:thenbo}
  g_n(\theta)-g_n(\theta_c)\ge C^{-1}(\theta-\theta_c)\quad \theta\in[\theta_c,\pi/2]\\
  g_n(\phi)-g_n(\theta_c)\le -C^{-1}(\theta_c-\phi)\quad \phi\in[\theta_c/2,\theta_c],
\end{eqnarray}
with $C$ an almost surely strictly positive random variable. Performing the integral then gives $|I_1 I_2|\lesssim n^{-2}$, which proves  \eqref{eq:corrstretched2}.

Finally, we prove \eqref{eq:stretched3}.
Up to exponentially small (in $n$) error terms,
\begin{eqnarray}
  \label{eq:I1I22}
  I_1 I_2=\int_{\theta_c}^{\pi/2} \cos\Big(\theta x-\frac\pi2(x\!\!\!\!\mod 4)\Big)
   \frac{\langle e_1,V_{<0}^\theta\rangle\langle V_{>2n+1}^\theta,e_1\rangle}{\langle
T_0^\theta\dots,T^\theta_n V_{>2n+1}^\theta,V_{<0}^\theta\rangle
   }\dd \theta \\
   \times
\int_{\theta_c/2}^{\theta_c}\cos\Big(\phi x-\frac\pi2(x\!\!\!\!\mod 4)\Big){\langle M_1^\phi T_1^\phi\dots,T^\phi_{n-1}M^\phi_{2n}e_1,e_1 \rangle}\dd \phi.
\end{eqnarray}
Let $\theta_-$ (resp. $\theta_+$) be the first zero  $\theta\mapsto \cos(\theta x-\frac\pi2(x\!\!\mod 4))$ to the left (resp. right) of  $\theta_c$, so that $\theta_+-\theta_-=\pi/ |x|$.
Because of the assumption $|\cos(x\theta_c-\frac\pi2(x \!\!\mod 4))|\ge {\delta}_2$, we have  $|\theta_\pm-\theta_c|\gtrsim {\delta_2}/|x|$. Split $I_1=I_1'+I_1'',I_2=I_2'+I_2''$, where in $I'_1$ (resp. $I'_2$) the integral is restricted to $[\theta_c,\theta_+]$  (resp. $[\theta_-,\theta_c]$).
The terms involving $I_1'' $ or $I_2''$ are negligible. For instance, using \eqref{tengoTr} and bounding the $1,1$  element of a positive matrix by its trace,
\begin{eqnarray}
  |I_1' I_2''|\lesssim (\theta_c)^{-4} \int_{\theta_c}^{\theta_+}e^{-n[g_n(\theta)-g_n(\theta_c)]}{\dd \theta}
  \int_{\theta_c/2}^{\theta_-}e^{n [g_n(\phi)-g_n(\theta_c)]}\dd \phi.
\end{eqnarray}
Using \eqref{eq:thenbo} and the conditions $(\theta_c-\theta_-)\gtrsim 1/|x|$,  $|x|\le \delta_1 n$, we see that $|I_1' I_2''|\le  n^{-2}e^{-\delta_2/(\delta_1C)}$. Similarly, one proves that $|I_1'' I_2'|,|I_1'' I_2''|\le  n^{-2}e^{-\delta_2/(\delta_1C)}$.

Finally, we bound $I_1' I_2'$ from below.
Since $\theta\mapsto \cos(x\theta-\frac\pi2(x \!\!\mod 4))$ does not change sign in $(\theta_-,\theta_+)$, restricting the integration to $\theta\in [\theta_c,\theta_c+(\theta_+-\theta_c)/2]$ and $\phi\in [\theta_c-(\theta_c-\theta_-)/2,\theta_c]$ we have using \eqref{eq:Ni2} and \eqref{eq:thenbo} and assuming that $\delta_1\ll\delta_2$
\begin{equs}
  I_1 I_2 &\gtrsim \delta_2^2 (\theta_c)^{4}\int_{\theta_c}^{\theta_c+\frac{\theta_+-\theta_c}2}e^{-n C^{-1}(
  \theta-\theta_c)}\dd \theta \int_{\theta_c-\frac{\theta_c-\theta_-}2}
  ^{\theta_c}e^{-n C^{-1}(\theta_c-\phi)}\dd \phi\\
  &\gtrsim \frac{\delta_2^2 }{n^2} \gg \frac{e^{-\delta_2/(\delta_1C)}}{n^2}.
\end{equs}
where in the second inequality we used $|\theta_\pm-\theta_c|n \gtrsim n\delta_2/|x|\gtrsim (\delta_2/\delta_1)\gg 1$.

\subsection{Proof of Theorem \ref{th:congas}}
\label{sec:thcongas}

We start from formula \eqref{eq:FgammaFormula}. Since we work at $H_2 = 0$, we omit the latter in the notations. Observe that $\cL_\gamma(\theta)\eqdef \cL_\gamma(\theta,0) = L_\gamma(2 \sin(\theta))$ with notations of Remark \ref{rmk:Lgamma}. Using this latter Remark, we deduce in particular that $\cL_\gamma$ is continuous and strictly increasing on $[0, \pi/2]$ (so that its inverse $\cL_\gamma^{-1}$ is well defined). As a consequence $F_\gamma(H_1)$ can be rewritten, similarly to \eqref{formulaF0}, as
\begin{equ}
  \label{formulaF0gamma}
     F_\gamma(H_1)=
     \begin{cases}
       F_\gamma(0)\eqdef\frac1{\pi}    \int_0^{\pi/2}\cL_\gamma(\theta)d\theta & \text{if } H_1\le \log \gamma\\
       F_\gamma(0)+\frac1\pi\int_{{\log \gamma}}^{H_1}
       \cL_\gamma^{-1}(y) dy& \text{ if } \log \gamma\le H_1\le H_{c,\gamma}\\
       \frac{H_1}2 & \text{ if } H_1\ge H_{c,\gamma}
     \end{cases}.
   \end{equ}
with $H_{c,\gamma}$ defined in \eqref{Hstargamma}.
For $H_2=0$, the matrix $T^\theta_\gamma$ can be bounded (entry-wise) as
\begin{eqnarray}
  \label{eq:Tgammabounds}
  {(\gamma w_2)^2}\,T^\theta_{-}\le T^\theta_\gamma\le (1+C \theta^2) {(\gamma w_2)^2}\,T^\theta_+,
\end{eqnarray}
where $C$ is a  constant, depending only on $\gamma$ and on the bounds on the support of the disorder, and
\begin{eqnarray}
  T^\theta_-\eqdef \begin{pmatrix}
    1 & C^{-1}\theta\\
    C^{-1}\theta Z_\gamma& Z_\gamma,
                   \end{pmatrix}
  \quad
  T^\theta_+\eqdef
  \begin{pmatrix}
  1 & C\theta\\
    C\theta Z_\gamma& Z_\gamma
  \end{pmatrix}, \quad Z_\gamma\eqdef \frac1{\gamma^4}\Big(\frac{w'_2}{w_2}\Big)^2.
\end{eqnarray}
Note that
$\mathbb E(\log Z_\gamma)<0$ (because $\gamma>1$) and $\mathbb E(Z_1)>1$ (by Jensen's inequality, since $w_2,w_2'$ are independent and non-degenerate).
From \eqref{eq:Tgammabounds}, it holds
\begin{eqnarray}
  \label{eq:LL}
  \log \gamma+\frac12{\rm Lyap}(T^\theta_-)\le \cL_\gamma(\theta)\le  \log \gamma+\frac12{\rm Lyap}(T^\theta_+)+O(\theta^2).
\end{eqnarray}

Define $\gamma_c$ as in \eqref{gammac}, so that  $\mathbb E(Z_{\gamma})\le1$ iff $\gamma\ge\gamma_c$.
To get the asymptotics of $F_\gamma(H_1)$ for $H_1\searrow \log \gamma$, we need the asymptotics of
the top Lyapunov exponent of
 $ \begin{pmatrix}
  1 & \eps\\
    \eps Z_\gamma& Z_\gamma
\end{pmatrix}
$
for $\eps>0$ small. Random matrices of this kind show up surprisingly often in mathematical physics, and the behavior of their Lyapunov exponent for $\eps\to0$ was predicted by  Derrida and Hilhorst  \cite{DH}. We will rely on rigorous versions of the Derrida-Hilhorst result, that we collect here:
\begin{theorem}
  \label{th:DH}
  Let $L(\epsilon)$ be the top Lyapunov exponent of the random matrix
$
  \begin{pmatrix}
  1 & \eps\\
    \eps Z& Z
  \end{pmatrix}
$
with $\eps>0$ and $Z$ a non-negative, non-degenerate random variable, compactly supported on $(0,\infty)$. Assume that $\mathbb E(\log Z)<0$ and define   $\alpha$ as the unique strictly positive value  such that $\mathbb E(Z^\alpha)=1$, or $\alpha=\infty$ if no strictly positive solution exists. The following holds as $\eps\to0$:
\begin{itemize}
\item If  $\alpha\in(0,1)$ (i.e., if $\mathbb E(Z)>1$)  and if the law of $\log Z$ has a  $C^1$ density, then
$    L(\eps)\sim C \eps^{2\alpha}.
$
\item If $\alpha\in(1,\infty]$ then
 $   L(\eps)\sim C \eps^{2}.$
\item If $\alpha=1$ then
$    L(\eps)=\eps^{2+o(1)}.
$\end{itemize}
In all cases, $C$ denotes a positive finite constant depending on the law of $Z$.
\end{theorem}
The result for $\alpha\in (0,1)$ is due to
\cite{GGG}, and the one for $\alpha\ge 1$ to \cite{Havret} (we mention that \cite{Havret} provides a much more refined expansion beyond the dominant $\eps^2$ term).

Now we go back to the problem of obtaining the asymptotics of $\cL_\gamma(\theta),\theta\to0$. Define $\alpha=\alpha(\gamma)$ as in Theorem \ref{th:DH} with the random variable $Z$ given by $Z_\gamma$, that is, as the positive solution to
\begin{eqnarray}
  \label{eq:alphagamma}
  \mathbb E(Z_\gamma^{\alpha(\gamma)})=1, \quad \text{i.e.}\quad \gamma^{4\alpha(\gamma)}=\mathbb E((w_2)^{2\alpha(\gamma)})\mathbb E((w_2)^{-2\alpha(\gamma)}),
\end{eqnarray}
or $+\infty$ if no positive solution exists.
Note that  $\alpha(\gamma)\le 1$ iff $\gamma\le \gamma_c$ and $\alpha(\gamma)=\infty$ for $\gamma$ large enough, since $w_2$ is bounded away from zero and infinity, so that $Z_\gamma<1$ almost surely if $\gamma$ is large enough. It is also easy to see from \eqref{eq:alphagamma} that \eqref{eq:ag2} holds.
If $\gamma\ge \gamma_c$, so that $\alpha(\gamma)\ge 1$, Theorem \ref{th:DH} and \eqref{eq:LL} imply that
\begin{eqnarray} \label{e:AsympLgamma3/2}
  \cL_\gamma(\theta)-\log\gamma\asymp \theta^2.
\end{eqnarray}
Applying \eqref{formulaF0gamma} immediately gives \eqref{eq:o1} with $\beta(\gamma)=3/2$. If instead $\gamma<\gamma_c$ we cannot apply directly  Theorem \ref{th:DH}, because under Assumption \ref{ass:disordine} the law of $\log Z_\gamma$ does not necessarily have a density (if it does and the density is $C^1$, then we directly get \eqref{eq:o2}).
This difficulty can be easily bypassed.
Let $X$ be a random variable (independent from all the other randomness) supported on $[0,1]$, whose law admits a $C^\infty$ density, and $\delta>0$. Define $T^\theta_{\pm,\delta}$ like $T^\theta_\pm$,  with $Z_\gamma$ replaced by
$Z_\gamma e^{\pm \delta X}$. By mononotonicity, the Lyapunov exponent of $T^\theta_{+,\delta}$ (resp. $T^\theta_{-,\delta}$) is larger (smaller) than that for $\delta=0$. Also, $\log Z_\gamma\pm\delta X$ has a $C^\infty$ and compactly supported density. 
Finally, it is easy to realize that $\delta\mapsto \alpha_\pm(\gamma,\delta)$ (defined as the positive solution of \eqref{eq:alphagamma} with $Z_\gamma$ replaced by $Z_\gamma e^{\pm \delta X}$) are continuous in $\delta$. At this point we can apply Theorem \ref{th:DH} and deduce (playing with the arbitrariness of $\delta$) that
\begin{eqnarray} \label{e:AsympLgammasingular}
  \cL_\gamma(\theta)-\log \gamma\stackrel{\theta\to0}=\theta^{2\alpha(\gamma)+o(1)}.
\end{eqnarray}
   Applying \eqref{formulaF0gamma} immediately gives \eqref{eq:o1} with $\beta(\gamma)=1+1/(2\alpha(\gamma))$.

\subsection{Singular behaviour of $H_2\mapsto \Delta_\gamma(H_2)$}

   \label{sec:conj}
   Here we prove an integrated version of the conjecture
   \eqref{eq:conj}. 
Write
\begin{equation}
\tilde \cL_\gamma(H_2)\eqdef  \Delta_{\gamma} (H_2)-\log\gamma = L_{\gamma} (\iota \sinh (H_2))-\log\gamma,
\end{equation}
with the notation \eqref{eq:Lgammaz}.
Assume for simplicity that  the law of $\log w_2$ has a $C^1$ density with respect to Lebesgue's measure.
We know from Section \ref{sec:thcongas} that
\begin{equation}
  \tilde{L}_{\gamma} (\varepsilon) \eqdef L_{\gamma} (\varepsilon) - \log
  \gamma \stackrel{\varepsilon \searrow 0} = C \varepsilon^{2 \min(\alpha (\gamma),1)}, \varepsilon
  \in \mathbb{R}^+.
\end{equation}




 By Remark \ref{rmk:Lgamma}, $z \mapsto L_{\gamma} (z)$ is continuous
on $\overline{\mathbb{H}}$ and harmonic on $\mathbb{H}$. Since $L_\gamma(z)$ grows  sublinearly (logarithmically) for $|z|\to\infty$,  we can write for $\eps\searrow 0$
\begin{equation}
 \frac{2\varepsilon}{\pi} \int_{\mathbb{R_+}}
\frac{\tilde L_\gamma(\iota y)}{\varepsilon^2 + y^2} \dd y=\tilde{L}_{\gamma} (\varepsilon) \sim C\varepsilon^{2 \min(\alpha (\gamma),1)}
\end{equation}
so that
\begin{equation}
  C\varepsilon^{2 \min(\alpha (\gamma),1)}(1-2^{1-2\min(\alpha (\gamma),1)})
  \sim \tilde L_\gamma(\eps)-2\tilde L_\gamma(\eps/2)=-\int_{\R_+}\tilde L_\gamma(\iota\eps u)f(u)\dd u,
\end{equation}
with
\begin{equation}
   \quad f(u)=\frac{6}{\pi(1+u^2)(1+4u^2)}.
\end{equation}
This can be rewritten as
\begin{equation}
  \label{almostlaplace}
\mathbb E(\tilde\cL_\gamma( {\rm asinh}(H_2 X)))=  \mathbb E(\tilde L_\gamma(\iota H_2 X))\stackrel{H_2\searrow 0}\sim c_{\alpha(\gamma)} |H_2|^{2 \min(\alpha (\gamma),1)}
\end{equation}
with $c_\alpha > 0$ if $\alpha \in (0,1/2)$ and $c_\alpha < 0$ if $\alpha \in (1/2, 1]$, and where $X$ is the positive random variable with density equal $f$.
   
   \begin{appendix}
\section{The non-disordered dimer model}
\label{app:pure}
Here we collect a few standard facts about the non-disordered dimer model. First of all, when $w_1$ and $w_2=1$ are constants, the top Lyapunov exponent of \eqref{e:Matrixnew} is immediately computed as    \begin{equ}
     \cL_{pure}(\theta)=\log\left( w_1 \sin(\theta) + \sqrt{1 +  w_1^2 \sin(\theta)^2}\right)
   \end{equ}
   so that  $\cL_{pure}(\theta)\sim w_1 \theta $ for $\theta\downarrow 0$ and
   \begin{eqnarray}
     \label{solidpure}
   \cL_{pure}(\pi/2)-\cL_{pure}(\theta)\sim \frac {w_1}{2\sqrt{1+w_1^2}} (\theta-\pi/2)^2 \text{ as }  \theta\uparrow \pi/2.
   \end{eqnarray}
Defining
$\mathbf H_{c}:=\cL_{pure}(\pi/2)=\log(w_1+\sqrt{1+w_1^2})$,
\eqref{32pure} easily follows.

Secondly, the free energy $\mathbf F_\gamma$ of the non-disordered version of the model defined in Section \ref{sec:congas} can be written, using Kasteleyn's theory for general bipartite periodic graphs \cite{KOS},
as
\begin{multline}
  \mathbf F_\gamma(H_1,H_2)=\int_{\mathbb T_{H_1,H_2}}\log P(z_1,z_2)\frac{\dd z_1}{2\pi\iota z_1} \frac{\dd z_2}{2\pi\iota z_2},\quad \mathbb T_{H_1,H_2}=\{z_1,z_2\in \C:|z_i|=e^{2H_i}\}\\
  \label{eq:P(zw)}
  P(z_1,z_2)=z_1+\frac1{z_1}+w_1^2\Big(z_2+\frac1{z_2}\Big)+\gamma^2+\frac1{\gamma^2}+2w_1^2.
\end{multline}
The gaseous region $G$ of the plane $(H_1,H_2)$ is defined \cite{KOS} as the region including $(0,0)$ where $P$ has no zeros on $\mathbb T_{H_1,H_2}$. It is easy to see that,  as soon as $\gamma>1$, $G$ is non-empty and its intersection with the $H_2=0$ axis is the segment $(-\log \gamma,\log \gamma)$.  $\mathbf F_\gamma$ is affine on $G$.
\end{appendix}

\section*{Acknowledgments}
We are very grateful to Orph\'ee Collin and Giambattista Giacomin
for enlightening discussions and for suggesting Refs.
\cite{de2024scaling,Havret,GGG}.  This research was funded by the Austrian
Science Fund (FWF) 10.55776/F1002.  For open access purposes, the
authors have applied a CC BY public copyright license to any author
accepted manuscript version arising from this submission.

\printbibliography

\end{document}